\documentclass[11pt,reqno]{amsart}%
\usepackage[a4paper, margin=2.5cm]{geometry}%
\RequirePackage[colorlinks,linkcolor=blue,citecolor=blue,urlcolor=blue]{hyperref}
\usepackage[T1]{fontenc}%
\usepackage[utf8]{inputenc}%
\usepackage{lmodern,mathtools}
\usepackage{amsmath,amsfonts,amssymb,amsthm}
\usepackage{amsmath}
\usepackage{tikz}
\usepackage{dsfont}
\makeatletter
\renewcommand{\paragraph}{%
  \@startsection{paragraph}{4}{\z@}%
  {3.25ex \@plus 1ex \@minus .2ex}%
  {-1em}%
  {\normalfont\normalsize\itshape}}
\makeatother
\title[]{Improved well-posedness for the limit flow of differentiation of roots of polynomials}

\author{Charles Bertucci}
\address[CB]{CEREMADE, CNRS, Université Paris Dauphine-PSL, UMR 7534, 75016 Paris, France.}
\email{\url{mailto:bertucci@ceremade.dauphine.fr}}
\author{Valentin Pesce}
\address[VP]{CMAP, Ecole Polytechnique, UMR 7641, 91120 Palaiseau, France.}
\email{\url{mailto:valentin.pesce@polytechnique.edu}}

\date{Winter 2026, compiled \today}
\newtheorem{theorem}{Theorem}[section]%
\newtheorem{definition}[theorem]{Definition}
\newtheorem{proposition}[theorem]{Proposition}%
\newtheorem{lemma}[theorem]{Lemma}%
\newtheorem{remark}[theorem]{Remark}%

\let\ensembleNombre\mathbb
\newcommand\tab[1][0.5cm]{\hspace*{#1}}
\newcommand\mes{\mathcal P(\R)}
\newcommand\mesC{\mathcal P(\C)}
\newcommand\mesT{\mathcal P(\mathbb T)}
\newcommand\N{\ensembleNombre{N}}
\newcommand\Z{\ensembleNombre{Z}}

\newcommand\R{\ensembleNombre{R}}
\newcommand\C{\ensembleNombre{C}}
\newcommand\T{\ensembleNombre{T}}
\newcommand\dis{\displaystyle}

\DeclareMathOperator{\leb}{Leb} 
\DeclareMathOperator{\cotan}{cotan} 
\DeclareMathOperator{\divergence}{div} 
\DeclareMathOperator{\convexhull}{Conv} 
\DeclareMathOperator{\arcot}{arcot} 

\makeatletter
\def\@MRExtract#1 #2!{#1}     
\renewcommand{\MR}[1]{
  \xdef\@MRSTRIP{\@MRExtract#1 !}%
  \href{http://www.ams.org/mathscinet-getitem?mr=\@MRSTRIP}{MR-\@MRSTRIP}}
\makeatother

\numberwithin{equation}{section}

\keywords{Random polynomials, Partial differential equation, Viscosity solutions, Mean field limit.}

\begin{document}
\begin{abstract}
In this paper, we study the partial differential equation on the circle that was heuristically obtained by Steinerberger \cite{steinerberger2019nonlocal} on the real line and which represents the evolution of the density of the roots of polynomials under differentiation. After integrating the partial differential equation in question, we observe that it can be treated with the theory of viscosity solutions. This equation at hand is a non linear parabolic integro-differential equation which involves the elliptic operator called the half-Laplacian. Due to the singularity of the equation, we first restrict our study to strictly positive initial condition. We obtain a comparison principle for solutions of the primitive equation which yields uniqueness, existence, continuity with respect to initial condition. We then extend these results to general initial data. We finally prove that the system of particles on the circle introduced by Steinerberger indeed approximates the solution of the equation.
\end{abstract}
\maketitle
{\footnotesize\tableofcontents}
\section{Introduction}
\subsection{General introduction}
In this paper we study the limit partial differential equation (PDE) that was heuristically obtained by Steinerberg \cite{steinerberger2019nonlocal} in a dynamical approach to characterize the flow of differentiation. The heuristic is quite difficult to set going but Steinerberg points in his paper some explicit solutions of the PDE which come from orthogonal polynomials families. This motivates the studying of the limit PDE to make rigorous the heuristic of the convergence of particles. 
Moreover this problem has been surprisingly connected to free probability, which motivated several works in this direction \cite{hall2023zeros,kabluchko2022lee,jalowy2025zeros}.
\newline
\tab The problem was first studied for random roots of trigonometric polynomials because there is not a loss of roots after differentiation (and so at the level of the PDE in the trigonometric case the mass of a solution is preserved). In this case a connection with free probability was also studied \cite{kabluchko2021repeated}.
In \cite{kiselev2022}, Kiselev and Tan study the limit PDE on the circle and they prove the existence and uniqueness of smooth solutions of the PDE starting from a strictly positive smooth initial condition. They also obtain a notion of convergence of the system of particles towards the strong solution of the PDE under assumptions of smoothness. 
\newline
\tab We first briefly state already known results about the location of roots of polynomials after differentiation both for the sake completeness and to motivate the problem. We recall the heuristic of the proof of Steinerberg \cite{steinerberger2019nonlocal} to formally obtain the PDE that shall be studied.
\newline
\tab Then, the main goal of this chapter is to introduce a new kind of solutions of this PDE in the trigonometric case and to prove existence and uniqueness for this new notion of solutions. Thanks to this approach we could study the PDE for less regular initial conditions without the assumption of strict positivity in comparison with what was done in \cite{kiselev2022}.
We note that if we integrate the PDE, which means we look at the PDE satisfied by the primitive function of a solution of the PDE, we obtain a new PDE that can be studied using the theory of viscosity solutions. 
This primitive equation is a parabolic integro-differential partial differential equation.
The integro-differential operator that appears is known as the half Laplacian on the circle and satisfies an elliptic property which hints to look at viscosity solutions of this equation. However, this PDE has a singularity which makes difficult at first sight to state properly a comparison principle as we usually do in viscosity theory. To overcome this issue, we first change the PDE by truncating the singularity. At the level of the original PDE it physically says that we are looking at solution starting from a strictly positive initial condition as in \cite{kiselev2022}. For these truncated primitive equations we get a comparison principle and so uniqueness of the solution of the PDE by using standard arguments in viscosity theory solutions (\cite{crandall1992} for fully non linear PDE of second order and \cite{arisawa2006,arisawa2008,barles2008} for PDE with an integro-differential term). Furthermore we can show that starting from an appropriate initial condition, we can guarantee that truncated PDEs and the original PDE are the same. Then, we extend, establishing a comparison principle between sub and super solutions of the PDE, the results obtained for truncated PDEs.
\newline
\tab Later we go back to the system of particles that are the roots of trigonometric polynomials. We understand the solutions that we just defined as a limit of the system of particles as it was done in a similar approach to the Dyson equation for the eigenvalues of random matrices \cite{bertucci2022spectral,bertucci2025new}. We first prove a discrete comparison principle for the particles. Finally, we properly derive the viscosity solutions of the PDE as the limit of the system of particles. 
\subsection{Overview of the results of the chapter}
The goal of this paper is to study the PDE: \begin{equation}
\label{eq:randompolpdeintro}
\boxed{
\partial_t u+\frac{1}{\pi}\,\partial_\theta\left(\arctan\left(\cfrac{H[u]}{u}\right)\right)=0\,\,\text{in } (0,+\infty)\times\T,}
\end{equation}
where $\T=\R/2\pi\Z$ and $H[u]$ is the periodic Hilbert transform. In \cite{kiselev2022}, it was proved that if the initial condition of this PDE is smooth (in $H^s$ for $s>3/2$) and strictly positive there exists a unique strong solution of this PDE and some properties as regularization, long time behaviour of these solutions have been studied.
\newline
The PDE \eqref{eq:randompolpdeintro} is very singular and having a "good" notion of solutions to obtain uniqueness or existence is not clear at all. Moreover a regularizing property of \eqref{eq:randompolpdeintro} is not expected in general. Indeed, as we shall mention from the heuristic derivation of the PDE it is expected that if the initial condition of \eqref{eq:randompolpdeintro} is the Dirac mass $\delta_0$ then the solution of the PDE will contain a Dirac mass for $t\in[0,1]$ at least.
\newline
We introduce a new notion of solutions to study the PDE \eqref{eq:randompolpdeintro} using the theory of viscosity solutions. Thanks to this approach one can obtain existence and uniqueness of solutions of \eqref{eq:randompolpdeintro} without any assumptions of strict positivity or regularity of the initial condition. We integrate the PDE \eqref{eq:randompolpdeintro} to obtain a PDE satisfied by the primitive of a solution of \eqref{eq:randompolpdeintro}: 
\begin{equation}
\boxed{
\partial_t F(t,\theta)+\cfrac{1}{\pi}\left(\arctan\left( \cfrac{A_0[F(t,.)](\theta)}{(\partial_\theta F(t,\theta))_+}\right)+\cfrac{\pi}{2}\right)=0\,\, \text{in } (0,+\infty)\times \R,
\label{eqPrimitiveintro}}
\end{equation}
where $A_0$ is the so-called periodic half Laplacian.
This PDE is well settled in the viscosity solutions theory thanks to the elliptic property of the half Laplacian. Thanks to this notion of solution we can now study solutions of \eqref{eq:randompolpdeintro} starting for instance from $\delta_0(dx)$.
\newline
However, the singularity of the PDE makes difficult to properly establish comparison principles for this PDE. So we consider truncated modification of \eqref{eqPrimitiveintro} for which the comparison principle holds:
\begin{align}
\boxed{
\tag{$E_m$}
\partial_t F(t,\theta)+\cfrac{1}{\pi}\left(\arctan\left( \cfrac{A_0[F(t,.)](\theta)}{\max((\partial_\theta F(t,\theta))_+,m)}\right)+\cfrac{\pi}{2}\right)=0\,\, \text{in } (0,+\infty)\times \R,}
\label{eqPrimitiveminintro}
\end{align}
where $m>0$ is a fixed real number.
\newline
The consequence of the truncation of the PDE \eqref{eqPrimitiveintro} is that we have to consider solutions of \eqref{eq:randompolpdeintro} starting from strictly positive initial condition (or equivalently solutions of \eqref{eqPrimitiveintro} with strictly non-decreasing initial condition). Then, using the study of these truncated PDEs, we extend the study of solutions of \eqref{eq:randompolpdeintro} starting from initial condition that are not necessarily strictly positive.
For a suitable notion of viscosity solutions introduced in Section \ref{def:viscoeq} we prove existence and uniqueness of non-decreasing solutions of \eqref{eqPrimitiveintro}. 
\begin{theorem}
\label{thm:existenceintro}
Let $F_0$ be a right-continuous, non-decreasing function such that $F_0(x+2\pi)=F_0(x)+1$. Then there exists a bounded viscosity solution of \eqref{eqPrimitiveintro} with initial condition $F_0$ which is non-decreasing in space for every time. 
\end{theorem}
\begin{theorem}
\label{thm:uniquenessintro}
Let $F_0$ be a right-continuous, non-decreasing function such that $F_0(x+2\pi)=F_0(x)+1$. Then there exists at most one bounded, non-decreasing in space viscosity solution of \eqref{eqPrimitiveintro} with initial condition $F_0$.
\end{theorem}
We also prove a continuity of this notion of solutions with respect to the initial condition. 
\begin{theorem}
\label{thm:stabilityintro}
Assume that $(F^n_0)_{n\ge 1}$ is a sequence of non-decreasing right-continuous functions that satisfy $F^n_0(x+2\pi)=F^n_0(x)+1$ that converges pointwise toward a continuous function $F_0$. 
Then, if we now write $(F_n(t,x))_{t\ge 0,x\in\R}$ (resp. $(F_\infty(t,x))_{t\ge 0,x\in\R}$) the unique viscosity solution of \eqref{eqPrimitivePol} with initial condition $F^n_0$ (resp. $F_0$), we have the following convergence $$\sup_{t\ge 0}||F_n(t,.)-F_\infty(t,.)||_{L^\infty(\R)}\underset{n\to+\infty}{\longrightarrow} 0.$$
\end{theorem}
Finally, we establish the convergence of the empirical distribution $(\mu_N(t))_{t\ge 0}$ of the periodic system of particles $(\lambda_t^{i,N})_{i\in\Z,t\ge 0}$ introduced by Steinerberger in \cite{steinerberger2019nonlocal} toward a viscosity solution of \eqref{eqPrimitivePol} in our framework using PDE tools. Let us also mention that the same result was obtained by Kabluchko \cite{kabluchko2021repeated} using complex analysis techniques and free probability.
\begin{theorem}
Assume that the empirical measure $\mu_N^0$ of initial conditions defined by: $$\mu_N^0=\frac{1}{2N}\sum_{i=0}^{2N-1}\delta_{\lambda_0^i}$$ 
converges toward a measure $\mu_0\in\mesT$. Let $F_0(x):=\mu_0((a,x])\in \mathcal F_{2\pi}$ with $a\in\R$ a point that is not an atom of $\mu_0$ and of $\mu_N^0$ for all $N\ge 1$.
\newline
Then we have the uniform in time convergence $$ \sup_{0\le t\le T} W_1(\mu_N(t),\partial_x F(t,.))\underset{N\to\infty}{\longrightarrow} 0\,,$$
where $F$ is the unique viscosity solution of \eqref{eqPrimitivePol} with initial condition $F_0$, where $W_1$ denotes the $1$-Wasserstein distance.
\end{theorem}
\subsection{Structure of the chapter}
First, in Section \ref{sect:briefoverviewoftheproblem} we recall some well known results about random polynomials that motivate the problem and the study of the PDE \eqref{eq:randompolpdeintro}.
\newline
Then, in Section \ref{section:studypde} we study the PDE \eqref{eq:randompolpdeintro} by integrating it and introducing a notion of viscosity solutions for this primitive equation. 
\newline
Finally, in Section \ref{section:systemofparticles} we investigate how viscosity solutions of \eqref{eqPrimitiveintro} can be constructed as limits of the system of particles heuristically introduced by Steinerberger in \cite{steinerberger2019nonlocal}. We present a complete proof for the convergence of the system of particles toward the viscosity solution of \eqref{eqPrimitiveintro}.
\section{A brief overview of the problem}
\label{sect:briefoverviewoftheproblem}
\subsection{Location of the roots of the derivative of a polynomial}
The problem of the location of the roots of the derivative of a polynomial knowing the roots of this polynomial have been a very well studied problem in the literature. For instance, see the non exhaustive references \cite{curgus2004contraction,stoyanoff1925theoreme,szeg1939orthogonal} and references therein.
\newline
The most important result on this topic is the so-called Gauss-Lucas theorem. This theorem can be viewed as a generalisation in $\C$ of the Rolle theorem for polynomials. We quickly recall proofs of this result since similar arguments shall be used later on. 
For $P\in\C[X]$, let $\mathcal Z(P)$ be the set of the roots of $P$. We write $\convexhull(A)$ the convex hull of $A\subset \C$ which is the smallest convex of $\C$ that contains $A$.
\begin{theorem}[Gauss-Lucas theorem] 
\label{gauss lucas}
Let $P\in \C[X]$ be a non constant polynomial, then $\mathcal Z(P')\subset \convexhull(\mathcal Z(P))$.
\end{theorem}

We give two proofs of this result because these two proofs use arguments that shall be used later on. The first one is based on the fundamental identity: 
\begin{equation}
\label{P'/P decomp}
\cfrac{P'(X)}{P(X)}=\sum_{i=1}^n \cfrac{1}{X-x_j},
\end{equation}
for $P\in \C_n[X]$ (where $\C_n[X]$ is the set of polynomial of $\C[X]$ of degree at most $n$) whose roots are $(x_i)_{i=1}^n$ with multiplicity. This relation is the key ingredient in what shall follow to obtain the partial differential equation. 
\newline
The second proof gives an electrostatic point of view for the Theorem \ref{gauss lucas}. It makes a link between the roots of complex polynomials and potential theory. This point of view was a key ingredient to obtain results about random polynomials as we will mention in Section \ref{section random pol}.
\begin{proof}
$\bullet$ First proof. Let $x_1,...,x_k$ be the roots of $P$ and $z\in\mathcal Z(P')$. If $z$ is a multiple root of $P$ the result is obvious.  Now assume that $z\in\mathcal Z(P')\setminus\mathcal Z(P)$. Evaluating $$\cfrac{P'(X)}{P(X)}=\sum_{i=1}^k \cfrac{1}{X-x_j}$$ in $z$ yields: $$0=\sum_{i=1}^k \cfrac{1}{z-x_j}.$$
To express $z$ as a convex combination of the $z_j$ we write: 
\begin{align*}
0&=\sum_{j=1}^k \cfrac{1}{z-x_j}=\sum_{j=1}^k \cfrac{\overline z -\overline x_j}{|z-x_j|^2}=\sum_{j=1}^k \cfrac{z -x_j}{|z-x_j|^2}
\end{align*}
It gives: $$z=\cfrac{\dis\sum_{j=1}^k \cfrac{x_j}{|z-x_j|^2}}{\dis\sum_{j=1}^k \cfrac{1}{|z-x_j|^2}}\in\convexhull(\mathcal Z(P)).$$
$\bullet$ Second proof. Let $x_1,...,x_k$ be the roots of $P$ and assume that $P$ is monic. Consider the $x_i$ as positive electrical charges in $\C=\R^2$. The electric charge density created by these charges is $\rho=\sum_{i=1}^k\delta_{x_j}$. We can compute the electrical field $\overrightarrow{E}$ created by the charges using the Maxwell-Gauss law. Indeed, we have:
\begin{align*}
\divergence (\overrightarrow E)&=\cfrac{\rho}{\varepsilon_0}\\
\overrightarrow E&=-\overrightarrow {\nabla} V,
\end{align*}
where $V$ is the electric potential associated to $\overrightarrow E$ and $\varepsilon_0$ is the vacuum permittivity.
It yields the so-called Laplace equation $$-\Delta V=\cfrac{\rho}{\varepsilon_0}.$$
We recall that the fundamental solution of the Laplacian in $\R^2$ is given by $$g(x,y)=-\cfrac{1}{2\pi}\ln(|(x,y)|)$$ (with $|(x,y)|$ the canonic norm in $\R^2$ of $(x,y)$). More precisely we have that for all $a\in \R^2$ in the sense of distributions: $$-\Delta \left[-\ln (|(x,y)-a|)\right]=2\pi\delta_a.$$
Hence the solution of the Laplace equation is $$V(z)=-\cfrac{1}{2\pi\varepsilon_0}\sum_{i=1}^k\ln|z-x_i|=-\cfrac{\ln|P(z)|}{2\pi\varepsilon_0},\, z\in\C.$$
So the electric field $\overrightarrow E=-\overrightarrow\nabla V$ vanishes exactly where $P'$ vanishes. However since the charges are all positive, the electric field can not vanish outside the convex hull of the charges proving the result.
\end{proof}
 
\subsection{Known results about random polynomials}
\label{section random pol}
\paragraph{Notation} We introduce the notation of measure theory that shall be used in this section.
\newline
$\bullet$ For $\mathbb K=\R$ or $\C$ we write $\mathcal P(\mathbb K)$ the set of probability measures supported on $\mathbb K$.
\newline
$\bullet$ For $\mathbb K=\R$ or $\C$ and $\mu \in\mathcal P(\mathbb K)$ and $f$ a reasonable function we write $\mu(f)=\int_{\mathbb K} fd\mu$.
\newline
$\bullet$ Let $\mathbb K=\R$ or $\C$. We recall that we say that a sequence of probability measures $\mu_N$ on $\mathcal P(\mathbb K)$ converges weakly (or in law) towards $\mu\in \mathcal P(\mathbb K)$ if and only if for all $f: \mathbb K\to \R$ continuous and bounded functions we have $\mu_N(f)\underset{N\to+\infty}{\longrightarrow}\mu(f)$. In this case, we write $\mu_N\overunderset{\mathcal{L}}{N\to+\infty}{\longrightarrow}\mu$.
\newline
There are two natural ways to generate a random polynomial. We can sample its roots or its coefficients. These two constructions are linked by the Vieta formulas. 
\newline 
\tab A question that really emerged during the $20^{th}$ century was: given a polynomial in $\R[X]$ whose coefficient are independent and identically distributed, could we estimate the average number of real roots this random polynomial has? Up to the knowledge of the authors Bloch and  Pólya were the first one to try to answer this question in \cite{bloch1932roots}. Littlewood and Offord \cite{littlewood1938number} dealt with this question in particular cases as when the coefficients are uniform on $[0,1]$ or Bernoulli random variables. Then, an astonishing result was obtained by Kac by giving an exact formula to compute the expectation of real roots of such polynomials and computing asymptotically the expectation when the degree of the polynomial goes to $+\infty$ and when the coefficients are Gaussian \cite{kac1943average}. Actually, the result obtained by Kac is quite "universal" in the sense that for reasonable distributions this asymptotic still holds when the degree increases \cite{stevens1969average} and similar computations were also done for random trigonometric polynomials \cite{dunnage1966number,farmer2006crystallization} using the Kac approach. This result was improved by obtaining an expansion of this expectation for instance using a geometric interpretation of the Kac formula obtained by Edelman and Kostlan \cite{edelman1995many}. We refer to \cite{do2015real,do2023random} for recent works on this topic and references therein. We also refer to \cite{hough2009zeros} which is a general textbook about the study of the roots of standard models of random polynomials. 
\newline
\tab In this chapter, we shall focus on the problem of random polynomials when the roots are sample independently and following the same distribution. This question was more recently asked by Pemantle and Rivin in \cite{pemantle2013distribution}. More precisely, let $\mu\in \mesC$ be a probability measure on $\C$ and $(X_i)_{i\in\N}$ a sequence of independent and identically distributed random variables with law $\mu$, we define for all $n\in\N$ the random polynomial $P_n(X)=\prod_{i=1}^n (X-X_i)$ and the empirical measure of the roots of $P_n'$ as $$\nu_n:=\frac{1}{n-1}\sum_{z\in\C,\, P_n'(z)=0 }\delta_z.$$
By the law of large numbers the empirical measure of the roots of $P_n$ almost surely converges weakly towards $\mu$. The question raised by Pemantle and Rivin was: how does the empirical measure of the roots of $P_n'$ behave? Thanks to their simulations Pemantle and Rivin conjectured that $(\nu_n)_{n\in\N}$ also almost surely converges weakly towards $\mu$. They proved the result under a technical assumption that must be satisfied by the measure $\mu$. However, this assumption fails for certain measures such as the uniform measure on the unit circle. Then, a universal result was obtained by Kabluchko in \cite{kabluchko2015critical}. He proved that for all $\mu\in\mesC$ the empirical mean of the roots of $P_n'$ converges in probability towards $\mu$. This proof is based on the electrostatic interpretation of the zeros of polynomials as it was mentioned in the second proof of Theorem \ref{gauss lucas}. He used that for $f$ a non zero holomorphic function on $\C$ we have that $\Delta\ln|f(z)|=2\pi\sum_{z\in \C,\, f(z)=0}\delta_z$ in the sense of distributions. Using this formula yields that for all $\phi$ smooth and compactly supported functions
\begin{equation}
\label{eq:equationdistrib}
\cfrac{1}{2\pi n}\int_\C \ln (|L_n(z)|) \Delta \phi(z) d\leb(z)=\cfrac{1}{n} \sum_{z\in \C,\, P_n'(z)=0}\phi(z)-\cfrac{1}{n}\sum_{z\in \C,\, P_n(z)=0}\phi(z),
\end{equation}
with $L_n(z):=\frac{P_n'(z)}{P_n(z)}$. The technical part of the proof consists in proving that the left hand side of \eqref{eq:equationdistrib} goes to 0 in probability which implies the result obtained by Kabluchko. The key argument is the use of potential theory tools as the Poisson–Jensen formula. This result represents an important step in the proof of the conjecture of Pemantle and Rivin but the convergence obtained by Kabluchko is in probability and not almost surely. Finally, the conjecture was solved by Angst $\&$ al. in \cite{angst2024almost}. The proof is based on the same idea as Kabluchko but they combined it with an anti-concentration inequality \cite{galligo2024anti}. Using the same arguments, this result was extended by proving that the result holds when we look at the empirical measure of the roots of the $k^{th}$ derivatives of $P_n$ \cite{michelen2024almost,michelen2024zeros} (also when $k$ depends on $n$ but slow enough).
\newline
\tab To finish this section and for the sake of completeness we give a counterpart of the proof of the Pemantle and Rivin conjecture in the case of zeros of trigonometric polynomials since we shall work on this case later on. This proof is nearly similar with the proof of the Pemantle and Rivin conjecture in the particular case when $\mu\in\mes$.

\begin{proposition}
Let $\mu\in\mesT$ and $(X_i)_{i\in\N}$ be a family of independent and identically distributed random variables of law $\mu$. For $N\ge 1$, let us define $P_N(X)=\prod_{i=1}^{2N}\sin\left(\frac{X-X_i}{2}\right)$ the trigonometric polynomial whose roots are the $(X_i)_{1\le i\le 2N}$. 

Let $\nu_N=\frac{1}{2N}\sum_{\theta\in\mathbb T,\, P_N'(\theta)=0 }\delta_\theta$ be the empirical measure of the roots of $P_N'$. Then almost surely $$\nu_N\overunderset{\mathcal L}{N\to\infty}{\longrightarrow}\mu.$$
\end{proposition}

\begin{proof}
For all $N>0$ we relabel the roots of $P_N$ as follows. We write $k(N)$ the (random) number of distinct roots of $P_N$ and we introduce $(\alpha_i(N))_{1\le i\le k(N)}$ the (random) multiplicity of the respective distinct roots of $P_N$. We first order the roots of $P_N$ 
\begin{align*}0\le X_{1,1}=...=X_{1,\alpha_1(N)}<X_{2,1}=...=X_{2,\alpha_2(N)}<...<X_{k(N),1}=...=&X_{k(N),\alpha_{k(N)}(N)}\\
&<2\pi.
\end{align*}
We can now relabel the roots of $P_N$ such that $\overline {X_1}:=X_{1,1}$, ..., $\overline {X_{\alpha_1(N)}}:=X_{1,\alpha_1(N)}$, $\overline {X_{\alpha_1(N)+1}}:=X_{2,1}$, ..., $\overline {X_{\alpha_1(N)+\alpha_2(N)}}:=X_{2,\alpha_2(N)}$, ..., $\overline {X_{2N}}:=X_{k(N),\alpha_{k(N)}(N)}$.
\newline
By the Rolle theorem we know that there are two kinds of roots of $P_N'$: the same roots as $P_N$ but with multiplicity $\alpha_i(N)-1$ for all $1\le i\le k(N)$ and roots (that are simple) obtained by using the Rolle theorem between $X_{i,\alpha_i(N)}<X_{i+1,1}$ for all $1\le i\le k(N)-1$ and between $X_{k(n),\alpha_{k(N)}(N)}<X_{1,1}+2\pi$.
For all $N>0$, let $(Y_i^N)_{1\le i\le 2N}$ be the $2N$ roots of $P_N'$ such that $Y_i^N$ is either $\overline{X_i}$ in the case of a multiple root or the unique root of $P_N'$ between $\overline{X_i}$ and $\overline{X_{i+1}}$ otherwise.
Let $f$ be a bounded Lipschitz function with Lipschitz constant $K$. Almost surely, for all $N$ we have: 
\begin{equation}
\begin{split}
\label{eq:lawoflargenumberscircle}
\left|\cfrac{1}{2N}\sum_{i=1}^{2N} f(\overline{X_i})-\cfrac{1}{2N}\sum_{i=1}^{2N} f(Y_i^N)\right|&\le\cfrac{1}{2N}\sum_{i=1}^{2N} \left|f(\overline{X_i})-f(Y_i^N)\right|\\
&\le\cfrac{K}{2N}\sum_{i=1}^{2N} \left|\overline{X_i}-Y_i^N\right|\\
& \le\cfrac{K}{2N}\sum_{i=1}^{2N} \left[\overline{X_{i+1}}-\overline{X_{i}}\right],
\end{split}
\end{equation}
where $\overline{X_{2N+1}}=\overline{X_1}+2\pi$ by convention. 
\newline
So we deduce that almost surely, for all $N$
$$\left|\cfrac{1}{2N}\sum_{i=1}^{2N} f(\overline{X_i})-\cfrac{1}{2N}\sum_{i=1}^{2N} f(Y_i^N)\right|\le \cfrac{K}{2N}\,(\overline{X_{2N+1}}-\overline{X_1})=\cfrac{\pi K}{N}.$$
By the strong law of large numbers for the $X_i$ we deduce that almost surely for all $f$ bounded Lipschitz function $$\cfrac{1}{2N}\sum_{i=1}^{2N} f(Y_i)\underset{N\to+\infty}{\longrightarrow}\int_\T f d\mu$$
which proves the result.
\end{proof}

\subsection{Heuristic of the PDE}
\label{section: heuristic pde}
In this section we recall how the PDE was formally obtained in \cite{steinerberger2019nonlocal}. We give the heuristic of the proof for polynomials and not trigonometric polynomials to focus on the main ideas. The heuristic of the proof for trigonometric polynomials is quite similar and can be found in \cite{kiselev2022}.
\newline
Let $(x_i)_{i\in\N}\in \R^\N$ be all distinct and write $p_N(X)=\prod_{i=1}^N(X-x_i)$.
Let $u\in \mes$ be a distribution that approximates the $x_i$ in the sense that $\frac{\sum_{i=1}^N \delta_{x_i}}{N} \sim u$. Let $y_i$ be the unique root of $p_N'$ in $(x_i,x_{i+1})$. The goal is to understand the flow created by the derivation and more precisely to look at $y_i-x_i$. Fix $m\in\N$.
As for the Gauss-Lucas theorem, we start from the identity: $$0=\sum_{i=1}^N\cfrac{1}{y_m-x_i}.$$ 
Since we look at a mean field limit we naturally renormalise the sum by starting from the identity: $$0=\cfrac{1}{N}\sum_{i=1}^N\cfrac{1}{y_m-x_i}.$$ 
We separate the sum in two parts associating to a near interaction and a far away interaction: 

\begin{equation}
\label{eq: heuristic 1}
0=\cfrac{1}{N}\sum_{|x_i-y_m|\le N^{-1/2}}\cfrac{1}{y_m-x_i}+\cfrac{1}{N}\sum_{|x_i-y_m|> N^{-1/2}}\cfrac{1}{y_m-x_i}=:S_{1,N}+S_{2,N}.
\end{equation}
Since the distribution $u$ approximates the $x_i$ we formally have: 
\begin{equation}
\label{eq: formally 1}
\cfrac{1}{N}\sum_{|x_i-y_m|> N^{-1/2}}\cfrac{1}{y_m-x_i}\sim \underset {\varepsilon\to 0}{\lim}\int_{|x-y_m|>\varepsilon}\cfrac{1}{y_m-x}u(dx)=\pi H[u](y_m),
\end{equation}
where $H[u]:=\frac{1}{\pi}P.V.(1/x)\ast u$ is the real Hilbert transform of $u$ with $P.V.$ the principal value.
\newline
About $S_{1,N}$, the idea introduced by Steinerberger is to approximate the $x_i$ near $y_m$ by equally distributed points around $x_m$.
We want to construct points $\tilde x_i$ such that $$\int_{x_m}^{\tilde x_i} u(x)dx=\cfrac{i-m}{N}.$$
Formally, approximating $\int_{x_m}^{\tilde x_i} u(x)dx\sim u(x_m)(\tilde x_i-x_m)$, we set $\tilde x_i=x_m+\cfrac{i-m}{N u(x_m)}$.
\newline
Formally, we rewrite $S_{1,N}$ by changing the $x_i$ by these new points: 
\begin{equation}
\label{eq formallu 2}
\begin{split}
S_{1,N}=\cfrac{1}{N}\sum_{|x_i-y_m|<N^{-1/2}}\cfrac{1}{y_m-x_i}&\sim \cfrac{1}{N}\sum_{|\tilde x_i-y_m|<N^{-1/2}}\cfrac{1}{y_m-\tilde x_i}\\
&\sim \cfrac{1}{N}\sum_{|k|<N^{1/2}}\cfrac{1}{y_m-x_m-\cfrac{k}{Nu(x_m)}}\\
&\sim u(x_m)\sum_{k\in\Z}\cfrac{1}{(y_m-x_m)Nu(x_m)+k}.
\end{split}
\end{equation}
We recall the Euler identity for the function $\cotan:=\frac{\cos}{\sin}$: $$\cotan(z)=\sum_{k\in\Z}\cfrac{1}{z+k\pi}.$$
Using this identity and \eqref{eq formallu 2}, we get: 
\begin{equation}
\label{eq:formmally 3}
S_{1,N}\sim \pi u(x_m)\cotan(N\pi u(x_m)(y_m-x_m))
\end{equation}
Passing to the limit in \eqref{eq: heuristic 1} formally yields:
\begin{equation}
\label{eq:formmally 4}
0=\pi H[u](y_m)+\pi u(x_m)\cotan(N\pi u(x_m)(y_m-x_m)).
\end{equation}
Since $u$ is a density that approximates the $x_i$, we have $$u(x_m)(x_{m+1}-x_m)\approx\int_{x_m}^{x_{m+1}} u(y)dy=\cfrac{1}{N}$$ which gives that $$0\le y_m-x_m\le x_{m+1}-x_m \le \cfrac{1}{Nu(x_m)},$$ and so $$0\le N\pi u(x_m)(y_m-x_m)\le \pi.$$ 
The formula \eqref{eq:formmally 4} yields a flow created by the derivation: $$y_m-x_m=\cfrac{1}{\pi Nu(x_m)}\arcot\left(-\cfrac{H[u](x_m)}{u(y_m)}\right),$$
where $\arcot$ is the reciprocal function of $\cotan :(0,\pi)\to \R$.
Using the identity $\arcot(-\theta)=\arctan(\theta)+\pi/2$ for all $\theta\in\R$, the flow creating by the derivation is
\begin{equation}
\label{eq: heuristic speed}
y_m-x_m=\cfrac{1}{\pi Nu(x_m)}\left(\arctan\left(\cfrac{H[u](y_m)}{u(x_m)}\right)+\cfrac{\pi}{2}\right).
\end{equation}
Taking $\Delta t=1/N$ as time scale, \eqref{eq: heuristic speed} gives that the derivation creates a macroscopic flux at speed: $$v(x_m)=\cfrac{1}{\pi u(x_m)}\left(\arctan\left(\cfrac{H[u](y_m)}{u(x_m)}\right)+\cfrac{\pi}{2}\right).$$
It gives at a macroscopic level the following conservative equation: 
\begin{equation}
\label{eq: heuristic equation}
\partial_t u+\partial_x(uv)=\partial_t u+\cfrac{1}{\pi}\,\partial_x\left(\arctan\left(\cfrac{H[u]}{u}\right)\right)=0.
\end{equation}
Let us also mention that in \cite{steinerberger2019nonlocal} Steinerberger also finds explicit solutions for the PDE \eqref{eq: heuristic equation} that can be constructed using well known results on orthogonal polynomials.
Another example starting from the degenerate condition $x_i=0$ for all $i\in\N$ gives that a solution of the PDE \eqref{eq: heuristic equation} starting from $\delta_0(dx)$ should be $u(t,dx)=(1-t)\delta_0(dx)$. 
\section{Study of the PDE and primitive equation}
\label{section:studypde}
We shall study the following PDE which is the analogous of \eqref{eq: heuristic equation} for trigonometric polynomials: 
\begin{equation}
\label{eq:randompolpde}
\boxed{
\partial_t u+\frac{1}{\pi}\,\partial_\theta\left(\arctan\left(\cfrac{H[u]}{u}\right)\right)=0\,\,\text{in } (0,+\infty)\times\T,}
\end{equation}
where $\T=\R/2\pi\Z$ and $$H[u]:=\frac{1}{2\pi}P.V.\cotan\left(\frac{.}{2}\right)\ast u$$ is the periodic Hilbert transform with $P.V.$ the principal value. If $u$ is smooth enough we can more explicitly write $H[u](y)=\frac{1}{2\pi}P.V.\int_\T\cotan(\frac{y-x}{2})(u(x)-u(y))dx$. Since the mass of a solution of \eqref{eq:randompolpde} is preserved, we shall look at solution $u$ of \eqref{eq:randompolpde} such that $u(t,.)\in\mesT$ for every $t\ge 0$. The existence, regularity and long time behaviour of solutions of \eqref{eq:randompolpde} starting from a smooth initial condition have been studied in \cite{kiselev2022}. The main idea of this section is to integrate the equation \eqref{eq:randompolpde} to obtain what we call the "primitive" equation of \eqref{eq:randompolpde}. This new equation shall be treated with the theory of viscosity solutions \cite{crandall1992,barles2008,arisawa2006,arisawa2008} 
to define a new kind of solutions of \eqref{eq:randompolpde}. This allows us to improve the existing results of existence and uniqueness of solutions of \eqref{eq:randompolpde} namely by allowing us to consider more general initial conditions that are not necessarily strictly positive.
\subsection{Notation}
We introduce the notation for the functional spaces that shall be used in the following parts.
\newline
$\bullet$ A function $f:\R\to\R$ is an element of $C^k(\T)$ for $k\in[0,\infty]$ (resp. $L^p(\T)$ for $p\in[1,\infty]$) if $f$ is a $2\pi$ periodic function of class $C^k$ (resp. if $f$ is a $2\pi$ periodic function in $L^p([0,2\pi])$). We write $\mathcal F(\T)$ for the space of functions on the circle that means the functions from $\R$ to $\R$ that are $2\pi$ periodic.
\newline
$\bullet$ Let $\mathcal F_{2\pi}$ (resp. $\mathcal F_{t,2\pi}$) be the space of functions $F$ defined on $\R$ (resp. $\R^+\times \R$) for which there exists $a\in\R$ such that for all $x\in\R$: $F(x+2\pi)=F(x)+a$ (resp. for all $t\ge 0$, for all $x\in\R$, we have $F(t,x+2\pi)=F(t,x)+a$).
\newline
$\bullet$ We say that $F\in \mathcal F_{2\pi}$ (resp. $\mathcal F_{t,2\pi}$) is bounded if there exists exists $C>0$ such that $F$ is bounded by $C$ on $[0,2\pi]$ (resp. for all $T>0$, $\sup_{t\in[0,T],x\in[0,2\pi]}|F(t,x)|<+\infty$). 
\newline
$\bullet$ Let $C_x^{1,1}(\R^+\times\R)$ be the space of functions $f(t,x)$ defined from $\R^+\times \R$ with values in $\R$ such that $f$ is a $C^1$ function and $\partial_x f$ is Lipschitz. 
\newline
$\bullet$ Let $C^{1,1}_{2\pi}$ be the space of functions $F\in C^1(\R,\R)$ such that $F'$ is a Lipschitz function and for which there exists $a\in\R$ such that for all $x\in\R$, we have $F(x+2\pi)=F(x)+a$.
\newline
$\bullet$ Let $C^{1,1}_{t,2\pi}$ be the space of functions of $F\in C_x^{1,1}(\R^+\times\R)$ , for which there exists $a\in\R$ such that for all $t\ge 0$, for all $x\in\R$, we have $F(t,x+2\pi)=F(t,x)+a$. 
\newline
$\bullet$ For a locally bounded function $F$ on a subset $\Omega$ of $\R^d$, we define for all $x\in\Omega$
$$F_*(x)=\underset{y\to x}{\liminf}\, F(y),\,\,F^*(x)=\underset{y\to x}{\limsup}\, F(y).$$
\subsection{Primitive equation} 
\label{subsection:Primitiveequation}
In \cite{bertucci2022spectral,bertucci2025new}, authors studied what is called the Dyson equation on $\R$ or $\T$ namely 
\begin{equation}
\label{dysoneq}
\partial_t u+\partial_x (uH[u])=0.
\end{equation} 
These PDE appear as limits of the empirical mean of the eigenvalues of continuous in time models of large random matrices introduced by Dyson in 1962 \cite{dyson1962}. The main idea in \cite{bertucci2022spectral,bertucci2025new} is to integrate the PDE that is to look at the PDE satisfied by the cumulative distribution function of a solution of \eqref{dysoneq}. This PDE can be treated with the theory of viscosity solutions. We follow this idea to deal with \eqref{eq:randompolpde}. 
\newline
Formally, integrating \eqref{eq:randompolpde} with respect to $\theta$ gives:
\begin{equation}
\partial_t F(t,\theta)+\cfrac{1}{\pi}\left(\arctan\left( \cfrac{H[\partial_\theta F(t,.)](\theta)}{(\partial_\theta F(t,\theta))_+}\right)+\cfrac{\pi}{2}\right)=0\,\, \text{in } (0,+\infty)\times \R.
\label{Primitive1}
\end{equation}
We put $(\partial_\theta F(t,\theta))_+$ instead of $\partial_\theta F(t,\theta)$ to have a non negative term in the denominator since we shall look at solutions of \eqref{Primitive1} which are non-decreasing in the space variable as primitives of measures. This way the PDE \eqref{Primitive1} is elliptic as we shall explain.
We choose the constant of integration with respect to what was obtained in the heuristic of the proof in \eqref{eq: heuristic speed}.
\newline
We introduce the operator $A_0$ defined on $C_{2\pi}^{1,1}$ and called the half Laplacian defined as:
\begin{equation}
A_0[f](\theta)=\cfrac{1}{4\pi}\,P.V.\int_{-\pi}^{\pi}\cfrac{f(\theta)-f(\theta-\theta')}{\sin^2\left(\frac{\theta'}{2}\right)}\,d\theta'=\cfrac{1}{8\pi}\int_{\T}\cfrac{2f(\theta)-f(\theta-\theta')-f(\theta+\theta')}{\sin^2\left(\frac{\theta'}{2}\right)}\,d\theta'.
\end{equation}
We notice that $A_0[f]=H[f']$ for $f\in\mathcal C_{2\pi}^{1,1}$.
\newline
Hence, we call primitive equation of \eqref{eq:randompolpde} the following PDE:
\begin{equation}
\boxed{
\partial_t F(t,\theta)+\cfrac{1}{\pi}\left(\arctan\left( \cfrac{A_0[F(t,.)](\theta)}{(\partial_\theta F(t,\theta))_+}\right)+\cfrac{\pi}{2}\right)=0\,\, \text{in } (0,+\infty)\times \R.
\label{eqPrimitivePol}}
\end{equation}

The operator $A_0$ satisfies a maximum principle: if $F\in C^{1,1}_{2\pi}$ has a maximum in $\theta_0$, then $A_0[F](\theta_0)\ge 0$. This motivates our approach by the viscosity solutions theory as in \cite{bertucci2022spectral,bertucci2024spectral,bertucci2025new}.

Starting from now, when we say that a periodic function $f$ has a strict maximum in $x$, it means that for all $y\ne x$ modulo $2\pi$ we have $f(y)<f(x)$ and similarly for a point of strict minimum.
\begin{definition} 
\label{def:viscoeq}\,
\newline
$\bullet$ An upper semi continuous (usc), non-decreasing in space function $F\in \mathcal F_{t,2\pi}$ is said to be a viscosity subsolution
of $~\eqref{eqPrimitivePol}$ if for any function $\phi\in C^{1,1}_{t,2\pi}$, $(t_0,\theta_0)\in (0,\infty)\times \R$ point of strict maximum of $F-\phi$ such that $\partial_\theta \phi(t_0,\theta_0)\neq 0$ the following holds: 
\begin{equation}
\partial_t \phi(t_0,\theta_0)+\cfrac{1}{\pi}\left(\arctan\left( \cfrac{A_0[\phi(t_0,.)](\theta_0)}{\partial_\theta \phi(t_0,\theta_0)}\right)+\cfrac{\pi}{2}\right)\le 0
\label{subsoleq}
\end{equation}
\newline
$\bullet$ A lower semi continuous (lsc), non-decreasing in space function $F\in \mathcal F_{t,2\pi}$ is said to be a viscosity supersolution
of $~\eqref{eqPrimitivePol}$ if for any function $\phi\in C^{1,1}_{t,2\pi}$, $(t_0,\theta_0)\in (0,\infty)\times \R$ point of strict minimum of $F-\phi$ such that $\partial_\theta \phi(t_0,\theta_0)\neq 0$ the following holds: 
\begin{equation}
\partial_t \phi(t_0,\theta_0)+\cfrac{1}{\pi}\left(\arctan\left( \cfrac{A_0[\phi(t_0,.)](\theta_0)}{\partial_\theta \phi(t_0,\theta_0)}\right)+\cfrac{\pi}{2}\right)\ge 0
\label{supersoleq}
\end{equation}
\newline
$\bullet$ A viscosity solution $F$ of $~\eqref{eqPrimitivePol}$ is a locally bounded, non-decreasing in space function such that $F$ is subsolution and $F_*$ is a supersolution.
\newline
$\bullet$ We say that the initial condition of $F$ is $F_0$ if $F^*(0,.)\le F_0^*$ and $F_*(0,.)\ge (F_0)_*$.
\newline
$\bullet$  A measure valued process $\mu\in C([0,+\infty),\mesT)$ is a viscosity solution to \eqref{eq:randompolpde} if $\mu=\partial_x F$ (in the sense of distributions) where $F$ is a viscosity solution to the primitive equation \eqref{eqPrimitivePol} which is non-decreasing in space, usc and that satisfies $F(t,x+2\pi)=F(t,x)+1$ for all $t\in\R^+$ and $x\in\R$. In this case, if $F_0$ is the initial condition of $F$, we also say that the initial condition of $\mu$ is $m_0=\partial_x F_0$.
\end{definition}

\begin{remark}
If $\mu\in C([0,+\infty),\mesT)$ is a viscosity solution of \eqref{eq:randompolpde} with initial condition $m_0$, we can check that $m_0=\mu(0)$.
Indeed, the viscosity initial condition implies that $F(t,\cdot)$ converges toward $F_0$ in $L^1$ over one period as $t$ goes to $0$. Passing to the limit in the distributional identity $\mu(t)=\partial_x F(t,\cdot)$ yields $\mu(0)=\partial_x F_0=m_0$.
\end{remark}
\begin{remark}
Let us precise that since $F$ and $\phi$ are in $\mathcal F_{t,2\pi}$ and $F-\phi$ has a point of global maxima (or minima) necessarily $F(t,.)-\phi(t,.)$ is a periodic function for every $t\ge 0$.
\end{remark}
\begin{remark}
Following the heuristic derivation of the PDE \eqref{eq:randompolpde} a solution of the PDE \eqref{eq:randompolpde} starting from $\delta_0(dx)$ should be in the real case $u(t,x)=(1-t)\delta_0(dx)$ for $0\le t<1$. One can easily check that the cumulative distribution function of $u$ is a viscosity solution of \eqref{eqPrimitivePol} using the real line analogue of Definition \ref{def:viscoeq}.
\end{remark}
We have to introduce a weak notion of solutions for \eqref{eq:randompolpde} since we do not expect any regularizing property of the PDE \eqref{eq:randompolpde} for non smooth initial data. Now we can try to prove existence and uniqueness for viscosity solutions of \eqref{eq:randompolpde} starting from non smooth initial data as $\delta_0(dx)$. 
\newline
However, due to  another main difficulty of the PDE which is the singularity, the notion of viscosity solutions introduced in Definition \ref{def:viscoeq} is difficult to use in practice. Indeed, we have to impose at a point of maximum that $\partial_\theta\phi \ne 0$ to make sense of the quantities in the PDE. Due to this extra condition it is difficult to directly state properly a comparison principle for this notion of viscosity solutions and so to obtain uniqueness of the viscosity solutions. 
\newline
To overcome this difficulty we first restrict our study to solutions of \eqref{eq:randompolpde} that formally satisfies $u(t,x)\ge m$ for all $t\ge 0$, for all $x\in\R$ for a certain $m>0$. This hypothesis shall be called $\eqref{hypothesis min}$ in Definition \ref{def:hypomin}. Having in mind the so-called monotone principle which states that for a smooth solution of \eqref{eq:randompolpde} the minimum is non-decreasing in time we shall justify (even at the level of viscosity solutions) that it suffices to satisfy the hypothesis $\eqref{hypothesis min}$ at time 0 to satisfy it for any $t\ge 0$.
\newline
To properly state a comparison principle, we introduce a family of PDEs obtained by truncating the singularity. Then, we shall explain how the results of existence, uniqueness and continuity in the initial condition obtained at the level of these truncated PDEs can be used to obtain results for the initial PDE.
\newline
For $m>0$ we consider the following primitive equation that shall be called $(E_m)$: 
\begin{align}
\boxed{
\tag{$E_m$}
\partial_t F(t,\theta)+\cfrac{1}{\pi}\left(\arctan\left( \cfrac{A_0[F(t,.)](\theta)}{\max((\partial_\theta F(t,\theta))_+,m)}\right)+\cfrac{\pi}{2}\right)=0\,\, \text{in } (0,+\infty)\times \R}
\label{eqPrimitivemin}
\end{align}

\subsection{Viscosity solutions of the truncated PDEs}
\label{Subsection:viscosol}
Equations \eqref{eqPrimitivemin} shall be treated with the theory of viscosity solutions \cite{crandall1992,arisawa2006,barles2008}.
\begin{definition} 
\label{def:visco}
Fix $m>0$.
\newline
$\bullet$ An upper semi continuous (usc) function $F\in \mathcal F_{t,2\pi}$ is said to be a viscosity subsolution
of $~\eqref{eqPrimitivemin}$ if for any function $\phi\in C^{1,1}_{t,2\pi}$, $(t_0,\theta_0)\in (0,\infty)\times \R$ point of strict maximum of $F-\phi$ the following holds: 
\begin{equation}
\partial_t \phi(t_0,\theta_0)+\cfrac{1}{\pi}\left(\arctan\left( \cfrac{A_0[\phi(t_0,.)](\theta_0)}{\max((\partial_\theta \phi(t_0,\theta_0))_+,m)}\right)+\cfrac{\pi}{2}\right)\le 0
\label{subsolpol}
\end{equation}
\newline
$\bullet$ A lower semi continuous (lsc) function $F\in \mathcal F_{t,2\pi}$ is said to be a viscosity supersolution
of $~\eqref{eqPrimitivemin}$ if for any function $\phi\in C^{1,1}_{t,2\pi}$, $(t_0,\theta_0)\in (0,\infty)\times \R$ point of strict minimum of $F-\phi$ the following holds: 
\begin{equation}
\partial_t \phi(t_0,\theta_0)+\cfrac{1}{\pi}\left(\arctan\left( \cfrac{A_0[\phi(t_0,.)](\theta_0)}{\max((\partial_\theta \phi(t_0,\theta_0))_+,m)}\right)+\cfrac{\pi}{2}\right)\ge 0
\label{supersolpol}
\end{equation}
\newline
$\bullet$ A viscosity solution $F$ of $~\eqref{eqPrimitivemin}$ is a locally bounded function such that $F$ is subsolution and $F_*$ is a supersolution.
\newline
$\bullet$ We say that the initial condition of $F$ is $F_0$ if $F^*(0,.)\le F_0^*$ and $F_*(0,.)\ge (F_0)_*$.
\end{definition}

As usual in the theory of viscosity solutions for integro-differential equations, it is more convenient to work with the local reformulation. 
We define a local and non local operator associated to $A_0$ as in \cite{arisawa2006,barles2008,bertucci2022spectral}. For $\delta>0$ let $I_{1,\delta}$ and $I_{2,\delta}$ defined by: 
\begin{align*}
I_{1,\delta}[\phi](\theta)&=\int_{\T\cap|\theta'|\le \delta}\cfrac{2\phi(\theta)-\phi(\theta-\theta')-\phi(\theta+\theta')}{\sin^2(\theta'/2)}\cfrac{d\theta'}{8\pi}\\
I_{2,\delta}[\phi](\theta)&=\int_{\T\cap|\theta'|> \delta}\cfrac{2\phi(\theta)-\phi(\theta-\theta')-\phi(\theta+\theta')}{\sin^2(\theta'/2)}\cfrac{d\theta'}{8\pi}.
\end{align*}
This allows us to distinguish between the part of $A_0$ which acts as a differential operator ($I_{1,\delta}$) and the rest ($I_{2,\delta}$).
\begin{proposition}
\label{prop: visco comparison principle}
Fix $m>0$.
Let $F$ be a subsolution (resp. supersolution) of $\eqref{eqPrimitivemin}$. Then for all $\phi\in C^{2}_{t,2\pi}$, $\delta>0$ and $(t_0,\theta_0)\in (0,+\infty)\times \R$ such that $(F-\phi) (t_0,\theta_0)=0$ and $(F-\phi)(t,\theta)\le 0$ (resp. $\ge  0$) for any $(t,\theta)\in B((t_0,\theta_0),\delta)$, the following holds: $$\partial_t \phi(t_0,\theta_0)+\cfrac{1}{\pi}\left(\arctan\left( \cfrac{I_{1,\delta}[\phi(t_0,.)](\theta_0)+I_{2,\delta}[F(t_0,.)](\theta_0)}{\max((\partial_\theta \phi(t,\theta))_+,m)}\right)+\cfrac{\pi}{2}\right)\le 0\text{ (resp. }\ge 0).$$
\end{proposition}

\subsection{Comparison principles}
We now state and prove a comparison principle for $\eqref{eqPrimitivemin}$. The proof is similar with the proof of the comparison principle obtained in the similar approach to study the eigenvalues of random matrices \cite{bertucci2022spectral,bertucci2025new}. 
\begin{theorem}
\label{thm:comparison principlepol}Fix $m>0$.
Assume that $u\in \mathcal F_{t,2\pi}$ and $v\in \mathcal F_{t,2\pi}$ are respectively bounded viscosity subsolution and supersolution of $\eqref{eqPrimitivemin}$. If $u(0,.)\le v(0,.)$, then for all time $t$, $u(t,.)\le v(t,.)$.
\end{theorem}

\begin{proof}
Let $\gamma>0$, instead of considering $u$ we could consider $u_{\gamma}(t,\theta)=u(t,\theta)-\gamma t$ which is a $\gamma$ strict subsolution in the sense that if $\phi\in C^{1,1}_{t,2\pi}$, $\delta>0$ and $(t_0,\theta_0)\in (0,+\infty)\times \R$ are such that $(u_{\gamma}-\phi) (t_0,\theta_0)=0$ and $(u_{\gamma}-\phi)(t,\theta)\le 0$ then for any $(t,x)\in B((t_0,x_0),\delta)$, the following holds: $$\partial_t \phi(t,\theta)+\cfrac{1}{\pi}\left(\arctan\left( \cfrac{I_{1,\delta}[\phi(t_0,.)](\theta_0)+I_{2,\delta}[u_\gamma(t_0,.)](\theta_0)}{\max((\partial_\theta \phi(t,\theta))_+,m)}\right)+\cfrac{\pi}{2}\right)\le -\gamma <0.$$
Hence if we prove that for all $t,\theta$, $u_{\gamma}(t,\theta)\le v(t,\theta)$ for all $\gamma>0$ , then by taking the limit $\gamma\to 0$ we shall recover $u\le v$. Hence, we can assume $u$ is a $\gamma$ strict subsolution in the proof and not just a subsolution. 
\newline
We argue by contradiction and suppose that there exists $t_0>0$, $\theta_0\in \R $ such that $u(t_0,\theta_0)> v(t_0,\theta_0)$. Let $T>t_0$ and use the classical technique of doubling variables. Thanks to the hypothesis, there exists $\alpha>0$ such that for all $\varepsilon>0$:
$$\sup\left\{u(t,\theta)-v(s,\theta')-\cfrac{1}{2\varepsilon}(\theta-\theta')^2-\cfrac{1}{2\varepsilon}(t-s)^2, \, t,s\in [0,T], \, \theta,\theta'\in \R^2\right\}>\alpha.$$
We shall justify that this supremum is actually a maximum by an usual localisation argument. 
\newline
Since $u$ and $v$ are in $\mathcal F_{t,2\pi}$, there exists $a$ and $b$ such that for all $t\ge 0$, for all $x\in\R$ $u(t,x+2\pi)=u(t,x)+a$ and $v(t,x+2\pi)=v(t,x)+b$. Since $u(0,.)\le v(0,.)$ we necessarily have $a=b$ and so $u(t,.)-v(t,.)$ is periodic for every $t\ge 0$. By noticing that the function $\beta(\theta,\theta'):=(\theta-\theta')^2$ satisfy $\beta(\theta+2k\pi,\theta'+2k\pi)=\beta(\theta,\theta')$ for all $k\in\Z$, we obtain that 
\begin{equation}
\begin{split}
\sup	&\left\{u(t,\theta)-v(s,\theta')-\cfrac{1}{2\varepsilon}(\theta-\theta')^2-\cfrac{1}{2\varepsilon}(t-s)^2, \, t,s\in [0,T], \, \theta,\theta'\in \R^2\right\}=\\
&\sup\left\{u(t,\theta)-v(s,\theta')-\cfrac{1}{2\varepsilon}(\theta-\theta')^2-\cfrac{1}{2\varepsilon}(t-s)^2, \, t,s\in [0,T], \, \theta\in[0,2\pi],\theta'\in \R\right\}.
\end{split}
\end{equation}
Thanks to the space penalization and the bounded assumption on $u$ and $v$, there exists $M>0$ such that we can now write 
\begin{equation}
\label{eq:pointmax}
\begin{split}
\sup	&\left\{u(t,\theta)-v(s,\theta')-\cfrac{1}{2\varepsilon}(\theta-\theta')^2-\cfrac{1}{2\varepsilon}(t-s)^2, \, t,s\in [0,T], \, \theta,\theta'\in \R^2\right\}=\\
&\sup\left\{u(t,\theta)-v(s,\theta')-\cfrac{1}{2\varepsilon}(\theta-\theta')^2-\cfrac{1}{2\varepsilon}(t-s)^2, \, t,s\in [0,T], \, \theta\in[0,2\pi],|\theta-\theta'|\le M\right\}.
\end{split}
\end{equation}
Since now the supremum is on a compact space, we can consider $t^*,s^*,\theta^*,\theta'^*\in [0,T]^2\times \R^2$  a point of maximum of \eqref{eq:pointmax}. We can classically assume that for $\varepsilon$ small enough, $t^*$ and $s^*$ are positive thanks to the fact that $\alpha>0$.
\newline
We want to use as tests functions: $$\phi_1(t,\theta)=v(s^*,\theta'^*)+\cfrac{1}{2\varepsilon}(\theta-\theta'^*)^2+\cfrac{1}{2\varepsilon}(t-s^*)^2,$$ for $u$ and: $$\phi_2(s,\theta')=u(t^*,\theta^*)-\cfrac{1}{2\varepsilon}(\theta^*-\theta')^2-\cfrac{1}{2\varepsilon}(t^*-s)^2,$$ for $v$.
\newline
The issue is that these functions are not in $C^{1,1}_{t,2\pi}$, we shall modify them a bit.
\newline
Take a small $\delta>0$ that will be specified later on.  We can find a function $\tilde \phi_1$ such that $\tilde \phi_1$ is equal to $\phi_1$ in $B((t^*,\theta^*),\delta)$ and which is in $C^{1,1}_{t,2\pi}$. Indeed for $\delta<\pi/2$, by localisation we can construct $\tilde \phi_1$ such that $\tilde \phi_1$ is equal to $\phi_1$ in $B((t^*,\theta^*),\delta)$ and impose that $\tilde\phi_1(t,\theta^*+\pi)=\tilde\phi_1(t,\theta^*-\pi)+1$ for all $t$ to define $\tilde\phi_1(t,.)$ on $[\theta^*-\pi,\theta^*+\pi]$ and then extend periodically $\tilde\phi_1(t,.)$ for all $t>0$.
We can do the same for $\phi_2$. We still call $\phi_1$ and $\phi_2$ these modifications of $\phi_1$ and $\phi_2$ that are now in $C^{1,1}_{t,2\pi}$. 
\newline
By using $\phi_1$ and $\phi_2$ as test functions in the definition of subsolution and supersolution we have that: $$\cfrac{1}{\varepsilon}(t^*-s^*)+\cfrac{1}{\pi}\left(\arctan\left( \cfrac{I_{1,\delta}[\phi_1(t^*,.)](\theta^*)+I_{2,\delta}[u(t^*,.)](\theta^*)}{\max\left(\cfrac{1}{\varepsilon}(\theta^*-\theta'^*)_+,m\right)}\right)+\cfrac{\pi}{2}\right)\le -\gamma$$
and: $$\cfrac{1}{\varepsilon}(t^*-s^*)+\cfrac{1}{\pi}\left(\arctan\left(\cfrac{ I_{1,\delta}[\phi_2(s^*,.)](\theta'^*)+I_{2,\delta}[v(s^*,.)](\theta'^*)}{\max\left(\cfrac{1}{\varepsilon}(\theta^*-\theta'^*)_+,m\right)}\right)+\cfrac{\pi}{2}\right)\ge 0.$$
We subtract the first inequality to the second and we have: 
\begin{equation}
\label{ineqfinalpol}
\begin{split}
&\arctan\left(\cfrac{ I_{1,\delta}[\phi_2(s^*,.)](\theta'^*)+I_{2,\delta}[v(s^*,.)](\theta'^*)}{\max\left(\cfrac{1}{\varepsilon}(\theta^*-\theta'^*)_+,m\right)}\right)-\\
&\hspace{4cm}\arctan\left( \cfrac{I_{1,\delta}[\phi_1(t^*,.)](\theta^*)+I_{2,\delta}[u(t^*,.)](\theta^*)}{\max\left(\cfrac{1}{\varepsilon}(\theta^*-\theta'^*)_+,m\right)}\right)\ge \pi\gamma.
\end{split}
\end{equation}
Firstly, we look at the $I_{2,\delta}$ part: 
\begin{align*}
&I_{2,\delta}[v(s^*,.)](\theta'^*)-I_{2,\delta}[u(t^*,.)](\theta^*)=\int_{\T\cap|z|>\delta}\cfrac{dz}{8\pi}`\cfrac{1}{\sin^2(z/2)}\times\\
&\left[2v(s^*,\theta'^*)-v(s^*,\theta'^*+z)-v(s^*,\theta'^*-z)-(2u(t^*,\theta^*)-u(t^*,\theta^*+z)-u(t^*,\theta^*-z))\right].
\end{align*}
Using that $(t^*,s^*,\theta^*,\theta'^*)$ is a point of maximum we have that for all $z\in \R$: 
\begin{align*}
u(t^*,\theta^*)-v(s^*,\theta'^*)\ge u(t^*,\theta^*+z)-v(s^*,\theta'^*+z)\\
u(t^*,\theta^*)-v(s^*,\theta'^*)\ge u(t^*,\theta^*-z)-v(s^*,\theta'^*-z).
\end{align*}
Adding these two inequalities we see that the previous numerator is actually non positive. Hence we have: 
\begin{equation}
\label{visco comp 2}
I_{2,\delta}[v(s^*,.)](\theta'^*)-I_{2,\delta}[u(t^*,.)](\theta^*)\le 0.
\end{equation}
Then for the terms in $I_{1,\delta}$ we notice that since $\phi_1(t^*,.)$ (resp. $\phi_2(s^*,.)$) is bounded in $C^2$ by $\varepsilon^{-1}$ in $B(\theta'^*,\delta)$ (resp. $B(\theta^*,\delta)$), we have: 
\begin{equation}
\label{visco comp3}
I_{1,\delta}[\phi_2(s^*,.)](\theta'^*)-I_{1,\delta}[\phi_1(t^*,.)](\theta^*)\le  \cfrac{2\delta}{\varepsilon\pi}. 
\end{equation}
Using \eqref{ineqfinalpol}, \eqref{visco comp 2}, \eqref{visco comp3}, and the fact that $\arctan$ is one Lipschitz and non-decreasing:
\begin{equation}
\label{ineq:compprinciple}
\begin{split}
\pi\gamma&\le\arctan\left(\cfrac{ I_{1,\delta}[\phi_2(s^*,.)](\theta'^*)+I_{2,\delta}[v(s^*,.)](\theta'^*)}{\max\left(\cfrac{1}{\varepsilon}(\theta^*-\theta'^*)_+,m\right)}\right)-\\
&\hspace{4cm}\arctan\left( \cfrac{I_{1,\delta}[\phi_1(t^*,.)](\theta^*)+I_{2,\delta}[u(t^*,.)](\theta^*)}{\max\left(\cfrac{1}{\varepsilon}(\theta^*-\theta'^*)_+,m\right)}\right)\\
&\le \cfrac{ I_{1,\delta}[\phi_2(s^*,.)](\theta'^*)+I_{2,\delta}[v(s^*,.)](\theta'^*)}{\max\left(\cfrac{1}{\varepsilon}(\theta^*-\theta'^*)_+,m\right)}-\cfrac{I_{1,\delta}[\phi_1(t^*,.)](\theta^*)+I_{2,\delta}[u(t^*,.)](\theta^*)}{\max\left(\cfrac{1}{\varepsilon}(\theta^*-\theta'^*)_+,m\right)}\\
&\le \cfrac{2\delta}{\varepsilon\pi\max\left(\cfrac{1}{\varepsilon}(\theta^*-\theta'^*)_+,m\right)}\le \cfrac{2\delta}{\varepsilon m\pi}
\end{split}
\end{equation}
Let $\delta=\varepsilon^2$ and let $\varepsilon$ goes to $0$ to obtain a contradiction since $\gamma>0$.
\end{proof}
\begin{remark}
This comparison principle was obtained using the parabolicity of the PDE which means at the level of particles that they repeal each others. We shall obtain in Section \ref{sect:compaparticles} a comparison principle at the discrete level of the particles.
\end{remark}
\subsection{Consequences of comparison principles for truncated PDEs}
Using the comparison principle, we shall prove results of existence and uniqueness of viscosity solutions of $\eqref{eqPrimitivemin}$. We also prove a counterpart of the minimum principle for the viscosity solutions of $\eqref{eqPrimitivemin}$.
\subsubsection{Uniqueness of viscosity solutions of the truncated PDE}
We can state a uniqueness result for viscosity solutions of \eqref{eqPrimitivemin}.
\begin{proposition}
\label{prop:uniqueness}
Given $m>0$ and $F_0\in \mathcal F_{2\pi}$ a bounded non-decreasing right continuous function, there exists at most one viscosity solution of \eqref{eqPrimitivemin} with initial condition $F_0$.
\end{proposition}
\begin{proof}
Assume that there exists $F$ and $G$ which are viscosity solutions of $\eqref{eqPrimitivemin}$ with the same initial data.
Let $\varepsilon>0$. Since $F$ and $G$ have the same initial condition $F_0$ we get that for all $x\in\R$: 
$$F(0,x)=F^*(0,x)\le (F_0)^*(x)\le (F_0)_*(x+\varepsilon)\le G_*(0,x+\varepsilon) .$$
Using Theorem \ref{thm:comparison principlepol} with $u=F$ as a viscosity subsolution and $v_\varepsilon(t,x)=G_*(t,x+\varepsilon)$ as a supersolution to \eqref{eqPrimitivemin}, we get that for all $t\ge0$, $\forall x\in\R$: $$F(t,x)\le G_*(t,x+\varepsilon).$$
Letting $\varepsilon$ goes to $0$ yields that for all $t\ge 0$, $\forall x\in\R$, $$F(t,x)\le G(t,x).$$
By symmetry of $F$ and $G$ we get that $F=G$.
\end{proof}

\subsubsection{Monotone principle}
First let us define what replace the condition $ u\ge m$ for a $m>0$ at the level of a primitive of $u\in\mesT$.
\begin{definition}\label{def:hypomin} Fix $m>0$.
A function $F:\R\to \R$ (resp. $F:\R^+\times \R\to\R$) is said to satisfied the hypothesis \eqref{hypothesis min} if:
\begin{equation}
\label{hypothesis min}
\tag{$H_m$}
\begin{split}
&\forall h\ge 0,\,\forall\theta\in\R,\, F(\theta+h)-F(\theta)\ge mh \\
(\text{resp.\,}
&\forall t\ge 0,\,\forall h\ge 0,\,\forall\theta\in\R,\, F(t,\theta+h)-F(t,\theta)\ge mh).
\end{split}
\end{equation} 
\end{definition}
As explained, for a smooth solution $\mu$ of \eqref{eq:randompolpde} the minimum principle states that $t\mapsto\min_\T\mu(t,.)$ is non-decreasing. We state a counterpart of this result at the level of viscosity solutions of $\eqref{eqPrimitivemin}$. 
\begin{proposition}
\label{propmonotoneprinciple}
Fix $m>0$.
Let $F$ be a bounded viscosity solution of $\eqref{eqPrimitivemin}$ with initial condition $F_0$. Assume that $F_0 \in \mathcal F_{2\pi}$ satisfies $\eqref{hypothesis min}$ and is a bounded non-decreasing right continuous function. Then $F$ satisfies $\eqref{hypothesis min}$.
\end{proposition}

\begin{proof}
Fix $h\ge 0$. By hypothesis on $F(0,.)$, we get that for all $\varepsilon>0$, for all $x\in\R$: $$F(0,x)+mh\le F_0(x)+mh\le F_0(x+h)\le (F_0)_*(x+h+\varepsilon)\le F_*(0,x+h+\varepsilon).$$
Using Theorem \ref{thm:comparison principlepol} with $u=F+mh$ as a subsolution and $v(t,x)=F_*(t,x+h+\varepsilon)$ as a supersolution we get that for all $\varepsilon>0$, $\forall t\ge 0$, $\forall x\in\R$:
$$F(t,x)+mh\le F_*(t,x+h+\varepsilon).$$
Since $F$ is usc, letting $\varepsilon$ goes to 0 gives the result.
\end{proof}

\begin{remark}
\label{remark fonction test}
Suppose that $F$ satisfies the hypothesis \eqref{hypothesis min} for $m>0$. Then we can replace in the definition of viscosity sub and supersolution $\max((\partial_\theta\phi(t_0,\theta_0))_+,m)$ by $\partial_\theta\phi(t_0,\theta_0)$. Indeed, for any $t_0,h\ge 0$, if $\theta_0$ is a point of maximum of $F(t_0,.)-\phi(t_0,.)$ then: $$mh+F(t_0,\theta_0)-\phi(t_0,\theta_0+h)\le F(t_0,\theta_0+h)-\phi(t_0,\theta_0+h)\le F(t_0,\theta_0)-\phi(t_0,\theta_0)$$ where we first used that $F(t_0,.)$ satisfies \eqref{hypothesis min} and then the fact that $\theta_0$ is a point of maximum of $F(t_0,.)-\phi(t_0,.)$. Hence, we deduce that $\partial_\theta \phi(t_0,\theta_0)\ge m$ and the same holds for a point of minimum.
\newline
This explains why under the hypothesis \eqref{hypothesis min} to find viscosity solutions of $\eqref{eqPrimitivePol}$ in the sense of Definition \eqref{def:viscoeq} we can work with solutions of \eqref{eqPrimitivemin}.
\end{remark}

\begin{remark}
\label{remark: solution en m}
Assume that $F$ satisfies \eqref{hypothesis min} and is a viscosity solution of \eqref{eqPrimitivemin} for $m>0$. Then for all $0<m'\le m$, $F$ satisfies $(H_{m'})$ and is a viscosity solution of $(E_{m'})$ using Remark \ref{remark fonction test}. 
\end{remark}

\subsubsection{Existence of a viscosity solution}
In this section we prove that there exists a viscosity solution of \eqref{eqPrimitivemin} starting from an initial data that satisfies \eqref{hypothesis min} for a given $m>0$. We first prove the result assuming that the initial condition is continuous. In this case we prove that the viscosity solution is continuous for every time. This proof is quite standard in viscosity theory. We then prove the result under the only assumption that the initial condition is the cumulative distribution function of a measure on $\T$. This point is more technical since we have to pay attention to the fact that the initial condition is just usc and not continuous as in the previous case.
\newline
We first state a general lemma about viscosity solutions that shall be used in the proof of existence of a viscosity solution of \eqref{eqPrimitivemin}. We sketch the proof but a detailed proof can be found in \cite{barles2008}. 
\begin{lemma}[Stability result]
\label{lemma:stability}
Fix $m>0$. Let $(F_n)_{n\in\N}$ be a sequence of functions such that for all $n\in \N$ for all $t\ge 0,\,x\in\R$ $F_n(t,x+2\pi)=F_n(t,x)+1$ that is a uniformly locally bounded sequence meaning that $$\forall T>0,\, \sup_{n}\sup_{t\in[0,T],\,x\in[0,2\pi]}|F_n(t,x)|<+\infty,$$ such that they are viscosity subsolutions of \eqref{eqPrimitivemin} (resp. viscosity supersolutions). 
\newline
The function defined by $F^*(t_0,\theta_0):=\underset{n\to+\infty,\,t\to t_0,\,\theta\to\theta_0}{\limsup} F_n(t,\theta)$ (resp. $F_*(t_0,\theta_0):=
\newline 
\underset{n\to+\infty,\,t\to t_0,\,\theta\to\theta_0}{\liminf} F_n(t,\theta)$) is a viscosity subsolution (resp. a visocosity supersolution) of \eqref{eqPrimitivemin}.
\end{lemma}
\begin{proof}
Let $\phi\in C_{t,2\pi}^{1,1}$ and $(t_0,\theta_0)$ a point of strict maximum of $F^*-\phi$. For the moment let us consider a sequence $(t_n,\theta_n)$ such that $F_n-\phi$ has a point of maximum at $(t_n,\theta_n)$ and such that $(t_n,\theta_n)\underset{n\to+\infty}{\to}(t_0,\theta_0)$. As for $n\ge 0$, $F_n$ is a viscosity subsolution we get 
\begin{equation}
\label{lemma:stabineq}
\partial_t\phi(t_n,\theta_n)+\cfrac{1}{\pi}\left(\arctan\left(\cfrac{A_0[\phi(t_n,.)](\theta_n)}{\max((\partial_\theta \phi(t_n,\theta_n))_+,m)}\right)+\cfrac{\pi}{2}\right)\le 0.
\end{equation}
Since we have $$A_0[\phi(t_n,.)](\theta_n)=\cfrac{1}{4\pi}\int_{-\pi}^\pi\cfrac{\phi(t_n,\theta_n)-\phi(t_n,\theta_n-\theta)-\theta\partial_\theta\phi(t_n,\theta_n)}{\sin^2(\frac{\theta}{2})}\,\,d\theta$$ and $\phi$ is in $C_{t,2\pi}^{1,1}$, the dominated convergence theorem yields $$A_0[\phi(t_n,.)](\theta_n)\underset{n\to+\infty}{\longrightarrow}A_0[\phi(t_0,.)](\theta_0).$$
Passing to the limit in the inequality \eqref{lemma:stabineq} gives that $F^*$ is a viscosity subsolution. To make this proof complete, we should use the local formulation of viscosity solutions of \eqref{eqPrimitivemin} because the points of maximum $(t_n,\theta_n)$ as they are constructed can only be considered in a ball around $(t_0,\theta_0)$ and not globally.
The same argument holds for $F_*$.
\end{proof}
\begin{proposition}
\label{prop:existence}
Fix $m>0$ and let $F_0$ be a usc function that satisfies $\eqref{hypothesis min}$ and $F_0(x+2\pi)=F_0(x)+1$ for all $x\in\R$ and write $\mu_0=\partial_x F_0$ (in the sense of distribution).  
\newline
$\bullet$ If $F_0$ is continuous then there exists a continuous viscosity solution of \eqref{eqPrimitivemin} with initial condition $F_0$
\newline 
$\bullet$
Even if $F_0$ is not continuous there exists a viscosity solution of \eqref{eqPrimitivemin} with initial condition $F_0$. 
\end{proposition}

\begin{proof}
$\bullet$ Assume that $F_0$ is continuous. We start by regularizing the initial data. Let $(\rho_\varepsilon)_{\varepsilon>0}$ be an approximation of the unity on $\T$ and define $\mu_0^\varepsilon:=\mu_0\ast \rho_\varepsilon$ and $F_0^\varepsilon(x):=F_0(0)+\int_0^x\mu_0^\varepsilon(y)dy$.
Fix $\varepsilon>0$. Since $\mu_0^\varepsilon$ is smooth and satisfies $\min_\T\mu_0^\varepsilon\ge m$ we can consider a smooth strong solution $(\mu_t^\varepsilon)_{t\ge 0}$ of \eqref{eq:randompolpde} with initial condition 
$\mu_0^\varepsilon$ using Theorem 1.1 of \cite{kiselev2022}. From the smooth solution $\mu^\varepsilon$, we construct a smooth solution $F^\varepsilon$ of \eqref{eqPrimitivePol} such that for all time $t\ge 0$ $\partial_x F^\varepsilon(t,.)=\mu^\varepsilon(t,.)$, $F^\varepsilon(0,.)=F_0^\varepsilon$ and that satisfies \eqref{hypothesis min}. In particular, for all $\varepsilon>0$, $F^\varepsilon$ is a viscosity solution of \eqref{eqPrimitivemin} that satisfies \eqref{hypothesis min}. 
\newline
We define $F_1(t_0,x_0):=\underset{\varepsilon\to 0,t\to t_0,x\to x_0}{\limsup}F^\varepsilon(t,x)$ and
$F_2(t_0,x_0):=\underset{\varepsilon\to 0,t\to t_0,x\to x_0}{\liminf}F^\varepsilon(t,x)$.
\newline
By Lemma \ref{lemma:stability}, $F_1$ is a viscosity subsolution of \eqref{eqPrimitivemin} and $F_2$ is a viscosity supersolution of \eqref{eqPrimitivemin}. Moreover, they also both satisfy \eqref{hypothesis min}.
\newline
Now let us look at the initial condition. Since for all $\varepsilon>0$, $F^\varepsilon$ is a strong solution of \eqref{eqPrimitivePol}, we have that for all $t\ge 0,\, x\in\R$, $$\partial_t F^\varepsilon(t,x)+\cfrac{1}{\pi}\left(\arctan\left(\cfrac{A_0[F^\varepsilon(t,.)](x)}{\partial_x F^\varepsilon(t,x)}\right)+\cfrac{\pi}{2}\right)=0.$$
From this we deduce that for all $\varepsilon>0$ for all $t\ge 0,\,x\in \R$ $$\left|\partial_t F^\varepsilon(t,x)\right|\le 1.$$
The mean value theorem gives that for all $\varepsilon>0$, for all $t\ge 0,\, x\in\R$ $$\left|F^\varepsilon(t,x)-F^\varepsilon(0,x)\right|\le t.$$
Let us first mention that from this estimate we deduce that the family $F^\varepsilon$ is uniformly in $\varepsilon$ locally bounded.
Moreover, since $F_0$ is continuous, passing to the superior and inferior limit in the previous inequality gives that $x\in\R$ $F_1(0,x)= (F_0)^*(x)=F_0(x)$ and $F_2(0,x)=(F_0)_*(x)=F_0(x)$. As $F_1$ is a viscosity subsolution of \eqref{eqPrimitivemin} and $F_2$ is a viscosity supersolution of \eqref{eqPrimitivemin}, by the comparison principle of Theorem \ref{thm:comparison principlepol}, we obtain that for every $t\ge 0$ $F_1(t,.)\le F_2(t,.)$ and thus $F:=F_1=F_2$ is a continuous viscosity solution that satisfies $F(0,.)=F_0$.
\newline
$\bullet$ We first regularize $F_0$ by considering for all $\varepsilon>0$ $F_0^\varepsilon:=F_0\ast\rho_\varepsilon$ with $\rho_\varepsilon$ an approximation of the unity such that $\rho_\varepsilon$ is supported in $[-\varepsilon,\varepsilon]$. We notice that since $F_0$ satisfies \eqref{hypothesis min} then so is $F_0^\varepsilon$ for all $\varepsilon>0$. Using Theorem 1.1 of \cite{kiselev2022} we can consider $(F^\varepsilon(t,x))_{t\ge 0,x\in\R}$ a  strong smooth solution of \eqref{eqPrimitivePol} starting from $F^\varepsilon(0,.)$. The solution $F^\varepsilon$ satisfies \eqref{hypothesis min} by the monotone principle ($t\mapsto\min_{x\in\T}\partial_x F^\varepsilon(t,x)$ is non-decreasing). In particular, for all $\varepsilon>0$, $F^\varepsilon$ is a viscosity solution of \eqref{eqPrimitivemin} that satisfies \eqref{hypothesis min}. 
\newline
We define $F_1(t_0,x_0):=\underset{\varepsilon\to 0,t\to t_0,x\to x_0}{\limsup}F^\varepsilon(t,x)$ and
$F_2(t_0,x_0):=\underset{\varepsilon\to 0,t\to t_0,x\to x_0}{\liminf}F^\varepsilon(t,x)$.
\newline
We shall prove that $F_1$ is a viscosity solution of $\eqref{eqPrimitivemin}$ with initial condition $F_0$.
By definition of being a viscosity solution of $\eqref{eqPrimitivemin}$ we have to prove that $F_1$ is a viscosity subsolution of $\eqref{eqPrimitivemin}$, $(F_1)_*$ is a viscosity supersolution of $\eqref{eqPrimitivemin}$ and $F_1(0,.)\le F_0$ and $(F_1)_*(0,.)\ge (F_0)_*$.
\newline
By Lemma \ref{lemma:stability}, $F_1$ is a viscosity subsolution of \eqref{eqPrimitivemin} and $F_2$ is a viscosity supersolution of \eqref{eqPrimitivemin}. Moreover, they also both satisfy \eqref{hypothesis min}.
\newline
Now let us look at the initial condition. Since for all $\varepsilon>0$, $F^\varepsilon$ is a strong solution of \eqref{eqPrimitivePol}, we have that for all $t\ge 0,\, x\in\R$, $$\partial_t F^\varepsilon(t,x)+\cfrac{1}{\pi}\left(\arctan\left(\cfrac{A_0[F^\varepsilon(t,.)](x)}{\partial_x F^\varepsilon(t,x)}\right)+\cfrac{\pi}{2}\right)=0.$$
From this we deduce that for all $\varepsilon>0$ for all $t\ge 0,\,x\in \R$ $$\left|\partial_t F^\varepsilon(t,x)\right|\le 1.$$
The mean value theorem gives that for all $\varepsilon>0$ for all $t\ge 0,\,x\in \R$
\begin{equation}
\label{eq:inegalitefonctionrepartitionlim}
-t +F^\varepsilon(0,x)\le F^\varepsilon(t,x)\le t+F^\varepsilon(0,x).
\end{equation}
Let us first mention, as in the previous proof, that from this estimate, we deduce that the family $F^\varepsilon$ is uniformly in $\varepsilon$ locally bounded.
Then, recalling that for $x\in\R$ $$F_1(0,x)= \underset{t\to 0,x_0\to x,\varepsilon\to 0}{\limsup}F^\varepsilon(t,x_0).$$
and using \eqref{eq:inegalitefonctionrepartitionlim} yields that for $x\in\R$
$$F_1(0,x)\le\underset{t\to 0,x_0\to x,\varepsilon\to 0}{\limsup}[t+F^\varepsilon(0,x_0)]=\underset{x_0\to x,\varepsilon\to 0}{\limsup}F^\varepsilon_0(x_0).$$
Since $\rho_\varepsilon$ is supported on $[-\varepsilon,\varepsilon]$ we have $F^\varepsilon_0(x)\le \sup_{y\in[x-\varepsilon,x+\varepsilon]}F_0(y)$ and so $$F_1(0,.)\le (F_0)^*.$$
Doing the exact same argument with the liminf, we obtain $(F_0)_*\le F_2(0,.)$. 
\newline
To conclude about the initial condition it suffices to show that $(F_1)_*(0,.)=F_2(0,.)$. Since $F_1\ge F_2$ we directly have that $(F_1)_*(0,.)\ge F_2(0,.)$. We prove the reverse. Indeed using \eqref{eq:inegalitefonctionrepartitionlim} we get $$F_1(t,x)\le t+(F_0)^*(x).$$ This implies $$(F_1)_*(0,x)\le ((F_0)^*)_*(x)=(F_0)_*(x)$$ since $F_0$ is non-decreasing. So $(F_1)_*(0,x)\le (F_0)_*(x)\le F_2(0,x)$. 
\newline
In summary at time $0$ we proved: 
\begin{equation}
\label{eq:initialcondition}
\left\{
\begin{aligned}
F_1(0,.)&\le F_0\\
(F_1)_*(0,.)&\ge (F_0)_*\\
(F_1)_*(0,.)&=F_2(0,.).
\end{aligned}
\right.
\end{equation}
\newline 
It remains to prove that $(F_1)_*$ is a viscosity supersolution of \eqref{eqPrimitivemin}. To show this we shall prove that $(F_1)_*= F_2$. By definition, since $F_1\ge F_2$ we have $(F_1)_*\ge (F_2)_*=F_2$.
\newline
Using the initial condition \eqref{eq:initialcondition} for all $\gamma>0$, we have $$F_1(0,.)\le F_0(.)\le (F_0)_*(.+\gamma)\le (F_2)_*(0,.+\gamma).$$
As $F_1$ is a viscosity subsolution of \eqref{eqPrimitivemin} and $F_2(.,.+\gamma)$ is a viscosity supersolution of \eqref{eqPrimitivemin} by the comparison principle of Theorem \ref{thm:comparison principlepol}, we obtain that for every $\gamma>0$, for every $t\ge 0$ $F_1(t,.)\le F_2(t,.+\gamma)$. Let $\gamma$ goes to 0 to obtain that $(F_1)_*\le (F_2)_*=F_2$. 
 
\end{proof}

\subsubsection{Continuity in the initial condition}
\begin{proposition}
\label{prop:continuity}
Fix $m>0$. Assume that $(F^n_0)_{n\ge 1}$ is a sequence of non-decreasing usc functions that satisfy $F^n_0(x+2\pi)=F^n_0(x)+1$ and that all satisfy $\eqref{hypothesis min}$ and that converges pointwise toward a continuous function $F_0$ . 
Then, if we now define $(F_n(t,x))_{t\ge 0,x\in\R}$ (resp. $(F_\infty(t,x))_{t\ge 0,x\in\R}$) the unique viscosity solution of \eqref{eqPrimitivemin} with initial condition $F^n_0$ (resp. $F_0$), we have the following convergence $$\sup_{t\ge 0}||F_n(t,.)-F_\infty(t,.)||_{L^\infty(\R)}\underset{n\to+\infty}{\longrightarrow} 0.$$
\end{proposition}

\begin{proof}
By the Dini convergence theorem the convergence of $F^n_0$ towards $F_0$ is also uniform on $[0,2\pi]$ and so on $\R$.
\newline
We recall by Proposition \ref{prop:existence} that $F_\infty$ is continuous and $F_\infty(0,.)=F_0$. 
\newline
Fix $\varepsilon>0$ and consider an integer $N$ such that for all $n\ge N$ one has $$||F^n_0(.)-F_0(.)||_{L^\infty(\R)}\le \varepsilon.$$
To complete the proof we shall prove for $n\ge N$ we have $$\sup_{t\ge 0}||F_n(t,.)-F_\infty(t,.)||_{L^\infty(\R)}\le \varepsilon.$$
For $n\ge N$, we have $$F_n(0,.)\le F^n_0(.)\le F_0(.)+\varepsilon=F_\infty(0,.)+\varepsilon.$$
By the comparison principle we obtain that for every $n\ge N$ and $t\ge 0$ $$F_n(t,.)\le F_\infty(t,.)+\varepsilon.$$
For the other inequality, for $n\ge N$, we have $$-\varepsilon+F_0(.)\le F_0^n(.)$$ Passing to the infimum and since $F_\infty$ is continuous one has $$-\varepsilon+F_\infty(0,.)=-\varepsilon+F_0(.)\le (F_0^n)_*(.)\le (F_n)_*(0,.).$$
Again by the comparison principle we get that for all $n\ge N$, for all $t\ge 0$ $$-\varepsilon+F_\infty(t,.)\le (F_n)_*(t,.)\le F_n(t,.),$$ which proves the result.
\end{proof}

\subsection{Existence, uniqueness of viscosity solutions of the original equation}

\subsubsection{A comparison principle for solutions of \eqref{eqPrimitivePol} and a stability result}
We first state and prove a comparison principle for $\eqref{eqPrimitivePol}$ when one of the function satisfies $\eqref{hypothesis min}$ for an $m>0$.  
\begin{lemma}
\label{thm:comparison principlepoledpinitial}
Assume that $u\in \mathcal F_{t,2\pi}$ and $v\in \mathcal F_{t,2\pi}$ are respectively bounded non-decreasing, viscosity subsolution and supersolution of $\eqref{eqPrimitivePol}$. Assume moreover that $u$ or $v$ satisfied $\eqref{hypothesis min}$ for a $m>0$. If $u(0,.)\le v(0,.)$, then for all time $t$, $u(t,.)\le v(t,.)$.
\end{lemma}
\begin{proof}
Let us suppose that $u$ satisfied $\eqref{hypothesis min}$ since if it is $v$ the arguments are similar. We keep the notation and the arguments of Theorem \ref{thm:comparison principlepol}. We use the following test functions: $$\phi_1(t,\theta)=v(s^*,\theta'^*)+\cfrac{1}{2\varepsilon}(\theta-\theta'^*)^2+\cfrac{1}{2\varepsilon}(t-s^*)^2,$$ for $u$ and: $$\phi_2(s,\theta')=u(t^*,\theta^*)-\cfrac{1}{2\varepsilon}(\theta^*-\theta')^2-\cfrac{1}{2\varepsilon}(t^*-s)^2,$$ for $v$.
Since $u$ satisfied $\eqref{hypothesis min}$ and $(t,\theta)\to u(t,\theta)-\phi_1(t,\theta)$ has a point of maximum in $(t^*,\theta^*)$ we have $\partial_\theta \phi_1(t^*,\theta^*)\ge m>0$. So $\phi_1$ is an admissible test function for $u$. Moreover, we notice that $\partial_{\theta'} \phi_2(s^*,\theta'^*)=\partial_\theta \phi_1(t^*,\theta^*)=\frac{\theta^*-\theta'^*}{\varepsilon}$ and so $\phi_2$ is also an admissible test functions for $v$. 

Doing exactly the same proof as in Theorem \ref{thm:comparison principlepol}, we obtain the counterpart inequality of \eqref{ineq:compprinciple}:
\begin{equation}
\pi\gamma\le \cfrac{2\delta}{\varepsilon\pi\left(\frac{\theta^*-\theta'^*}{\varepsilon}\right)}\le \cfrac{2\delta}{\varepsilon \pi m}.
\end{equation}
Letting $\delta$ goes to 0 gives the same contradiction as in Theorem \ref{thm:comparison principlepol}.
\end{proof}
We can now extend this comparison principle between a sub and a super solution of \eqref{eqPrimitivePol}.

\begin{theorem}
\label{thm:compprinciplecomplet}
Assume that $u\in \mathcal F_{t,2\pi}$ and $v\in \mathcal F_{t,2\pi}$ are respectively bounded non-decreasing in space, viscosity subsolution and supersolution of $\eqref{eqPrimitivePol}$ that satisfy $u(t,x+2\pi)=u(t,x)+1$ for all $t\ge 0$, $x\in\R$ and the same for $v$. If $u(0,.)\le v(0,.)$, then for all time $t$, $u(t,.)\le v(t,.)$.
\end{theorem}

\begin{proof}
For all $\gamma>0$, we can find $w_0^\gamma$ smooth with $w_0^\gamma$ satisfies $(H_{m_\gamma})$ with $m_\gamma$ that goes to 0 when $\gamma$ goes to 0 and such that $$u(0,.)-\gamma\le w_0^\gamma\le v(0,.)+\gamma.$$
We consider $W^\gamma$ the smooth solution of \eqref{eqPrimitivePol} with initial condition $W^\gamma(0,.)=w_0^\gamma$ and that $(H_{m_\gamma})$ using Theorem $1.1$ of \cite{kiselev2022}.
Hence, we have $u(0,.)\le W^\gamma(0,.)+\gamma$ and $W^\gamma(0,.)-\gamma\le v(0,.)$. Applying twice Lemma \ref{thm:comparison principlepoledpinitial} yields for all $t\ge 0$ $$u(t,.)\le W^\gamma(t,.)+\gamma\le v(t,.)+2\gamma.$$
Letting $\gamma$ goes to 0 gives the result.
\end{proof}
We also state a stability result for the PDE \eqref{eqPrimitivePol}
\begin{lemma} 
\label{lemma:stabilitybis}
Let $(F_n)_{n\in\N}$ be a sequence of non-decreasing in space functions such that for all $n\in \N$ for all $t\ge 0,\,x\in\R$ $F_n(t,x+2\pi)=F_n(t,x)+1$ that is a uniformly locally bounded sequence meaning that $$\forall T>0,\, \sup_{n}\sup_{t\in[0,T],\,x\in[0,2\pi]}|F_n(t,x)|<+\infty,$$ such that they are viscosity subsolutions of \eqref{eqPrimitivePol} (resp. viscosity supersolutions). 
\newline
The function defined by $F^*(t_0,\theta_0):=\underset{n\to+\infty,\,t\to t_0,\,\theta\to\theta_0}{\limsup} F_n(t,\theta)$ (resp. $F_*(t_0,\theta_0):=
\newline 
\underset{n\to+\infty,\,t\to t_0,\,\theta\to\theta_0}{\liminf} F_n(t,\theta)$) is a viscosity subsolution (resp. a visocosity supersolution) of \eqref{eqPrimitivePol}.
\end{lemma}
\begin{proof}
We sketch the proof of the subsolution since the arguments are similar with the arguments of Lemma \ref{lemma:stability}. Let $\phi$ a smooth test functions such that $F^*-\phi$ has a strict point of maximum in $(t_0,\theta_0)$ and such that $\partial_\theta \phi(t_0,\theta_0)\ne 0$. Since $F^*$ is non-decreasing, we get $\partial_\theta \phi(t_0,\theta_0)>0$. For all $n\ge 1$ we can consider a point of maxima $(t_n,\theta_n)$ of $F_n-\phi$ on $[t_0/2,2t_0]\times \T$. Using the fact that $(t_0,\theta_0)$ is a strict point of maxima of $F^*-\phi$ we necessarily have that, up to extraction, $(t_n,\theta_n)$ converges toward $(t_0,\theta_0)$ and so the regularity of $\phi$ yields $\partial_\theta\phi(t_n,\theta_n)$ converges toward $\partial_\theta\phi(t_0,\theta_0)$. So, for $n$ large enough, $\partial_\theta\phi(t_n,\theta_n)>0$. Since, $F_n$ is a is a subsolution of \eqref{eqPrimitivePol} we get
\begin{equation}
\label{lemma:stabineqbis}
\partial_t\phi(t_n,\theta_n)+\cfrac{1}{\pi}\left(\arctan\left(\cfrac{A_0[\phi(t_n,.)](\theta_n)}{\partial_\theta \phi(t_n,\theta_n)}\right)+\cfrac{\pi}{2}\right)\le 0.
\end{equation}
Passing to the limit, using the dominated convergence theorem as in Lemma \ref{lemma:stability}, gives that $F^*$ is a viscosity subsolution.
\end{proof}
\subsubsection{Consequences of this comparison principle}

As a corollary of Theorem \ref{thm:compprinciplecomplet} we get the following result for the uniqueness of viscosity solutions of \eqref{eqPrimitivePol}.

\begin{theorem}
Let $F_0$ be a non-decreasing, usc function that satisfies $F_0(x+2\pi)=F_0(x)+1$. Then, there exists at most one bounded non-decreasing in space viscosity solution of \eqref{eqPrimitivePol} in $\mathcal F_{t,2\pi}$ with initial condition $F_0$.
\end{theorem}

\begin{proof}
Let $u\in\mathcal F_{t,2\pi}$ and $v\in \mathcal F_{t,2\pi}$ be two viscosity solutions of \eqref{eqPrimitivePol} with initial condition $F_0$. For every $h>0$ and $x\in\R$, we have $$u(0,.)\le F_0(.)\le (F_0)_*(.+h)\le v_*(0,.+h).$$
Applying for every $h>0$, the comparison principle between $u$ and $v_*(.,.+h)$ yields for every $t\ge 0$ for every $h>0$ $$u(t,x)\le v_*(t,x+h).$$ Letting $h$ goes to 0 gives $u\le v$ and exchanging the role of $u$ and $v$ implies $ u=v$, proving the uniqueness.
\end{proof}
Thanks to the comparison principle and the stability lemma, we can obtain the existence of solutions of \eqref{eqPrimitivePol} as in Proposition \ref{prop:existence}.
\begin{theorem}
Let $F_0$ be a non-decreasing, usc function that satisfies $F_0(x+2\pi)=F_0(x)+1$. Then there exists a viscosity solution of \eqref{eqPrimitivePol} with initial condition $F_0$. Moreover, if $F_0$ is continuous then $F$ is continuous. 
\end{theorem}

\begin{proof}
We sketch the proof since the arguments are similar with those of Proposition \ref{prop:existence}. The main difference is how to apply Theorem 1.1 of \cite{kiselev2022} because here the solution is, as before, not smooth but also can be constant in some regions. For every $1>\varepsilon>0$, $\delta>0$ we define $$F_0^{\varepsilon,\delta}:=(1-\varepsilon)\rho_\delta\ast F_0+\varepsilon \frac{x}{2\pi}\,,$$
where $\rho_\delta$ is an even approximation of the identity supported in $[-\delta,\delta]$.
The functions $F_0^{\varepsilon,\delta}$ are smooth and satisfy $H_{\frac{\varepsilon}{2\pi}}$. By Theorem $1.1$ of \cite{kiselev2022}, we can consider $F^{\varepsilon,\delta}$ a smooth solution of \eqref{eqPrimitivePol} with initial condition $F_0^{\varepsilon,\delta}$. We first have $\|\partial_t F^{\varepsilon,\delta}\|_\infty\le 1$ since $F^{\varepsilon,\delta}$ is a solution of $\eqref{eqPrimitivePol}$ and moreover, the comparison principle implies that for every $\gamma>0$ $$\sup_{t\in[0,T],x\in\R}|F^{\varepsilon,\delta}(t,x)-F^{\gamma,\delta}(t,x)|\le \|F_0^{\varepsilon,\delta}-F_0^{\gamma,\delta}\| \le C |\varepsilon-\gamma|,$$
where $C=\|F_0-\frac{.}{2\pi}\|$ is a constant independent of $\delta$. 
So the sequence $(F^{\varepsilon,\delta})_\varepsilon$ is a Cauchy sequence in $C([0,T]\times[0,2\pi]\to \R)$ and so we can construct $F^\delta$ its limit when $\varepsilon$ goes to $0$. By Lemma \ref{lemma:stabilitybis}, we obtain that $F^\delta$ is a continuous viscosity solution of \eqref{eqPrimitivePol} with initial condition $F_0\ast \rho_\delta$. The functions $F^\delta$ satisfy $\sup_{x\in\R} |F^\delta(t,x)-F^\delta(s,x)|\le |t-s|$. We now similarly prove, as in Proposition \ref{prop:existence} but using this time the comparison principle of Theorem \ref{thm:compprinciplecomplet}, that $$F_1(t_0,x_0):=\limsup_{\delta\to 0,t\to t_0,x\to x_0}F^\delta(t,x)$$ is a viscosity solution of \eqref{eqPrimitivePol} with initial condition $F_0$ which is continuous if $F_0$ is continuous.
\end{proof}

Finally, doing the exact same proof as in Proposition \eqref{prop:continuity}, using this time the comparison principle of Theorem \ref{thm:compprinciplecomplet}, we obtain the continuity in the initial condition of viscosity solutions of \eqref{eqPrimitivePol}.
\begin{proposition}
Assume that $(F^n_0)_{n\ge 1}$ is a sequence of non-decreasing usc functions that satisfy $F^n_0(x+2\pi)=F^n_0(x)+1$ that converges pointwise toward a continuous function $F_0$. 
Then, if we now write $(F_n(t,x))_{t\ge 0,x\in\R}$ (resp. $(F_\infty(t,x))_{t\ge 0,x\in\R}$) the unique viscosity solution of \eqref{eqPrimitivePol} with initial condition $F^n_0$ (resp. $F_0$), we have the following convergence $$\sup_{t\ge 0}||F_n(t,.)-F_\infty(t,.)||_{L^\infty(\R)}\underset{n\to+\infty}{\longrightarrow} 0.$$
\end{proposition}

\section{System of particles}
\label{section:systemofparticles}

\subsection{Reminders and notation for trigonometric polynomials}
We consider the space of real trigonometric polynomials with $2n$ distinct roots in $(-\pi,\pi]$ 
\begin{align*}\R_{2n,per}=\{ p\, | \, \exists(a_i,b_i)_{i=0}^n\in \R^{2n}, p(x)=\sum_{j=0}^n&a_j\cos(jx)+b_j\sin(jx),\\
& p\text{ has  }2n\,\text{distinct roots in }(-\pi,\pi] \} .
\end{align*}
We notice that if $p\in \R_{2n,per}$ then all the derivatives of $p$ are also in $\R_{2n,per}$ by the Rolle theorem.
For $p\in \R_{2n,per}$, let $(x_j)_{j=1}^{2n}$ be its $2n$ distinct roots. 
We recall that for $p\in \R_{2n,per}$, we can find a constant $c\in\R$ such that $$p(x)=c\prod_{j=1}^{2n}\sin\left(\cfrac{x-x_j}{2}\right).$$ 
\newline
This gives the following identity for $p\in\R_{2n,per}$, \begin{equation}
\label{Formula p'/p}
\cfrac{p'(x)}{p(x)}=\cfrac{1}{2}\sum_{i=1}^{2n}\cotan\left(\cfrac{x-x_j}{2}\right).
\end{equation}
Let us notice that this identity is the counterpart of the decomposition $$\cfrac{P'(X)}{P(X)}=\sum_{i=1}^{n}\cfrac{1}{X-x_j}, $$ for $P\in \C_n[X]$ such that $P(X)=\prod_{i=1}^n (X-x_j)$.

\subsection{Comparison principles on the particles}
\label{sect:compaparticles}
We first state a comparison principle for the roots of real polynomials before stating it for trigonometric polynomials.
\begin{proposition}[Discrete comparison principle]
\label{prop:comp principle}
Let $x_1<x_2<...<x_N$, $y_1<y_2<...<y_N$ be two families of $N$ real numbers and consider $P_N(X)=\prod_{i=1}^N(X-x_i)$ and $Q_N(X)=\prod_{i=1}^N(X-y_i)$. Let $x_1<x_1'<x_2<x_2'<...<x_{N-1}'$ the $N-1$ real roots of $P_N'$ and $y_1<y_1'<y_2<y_2'<...<y_{N-1}'$ the $N-1$ real roots of $Q_N'$. 
Assume that for all $1\le i\le N$, $x_i\le y_i$, then for all $1\le i\le N-1$, $x_i'\le y_i'$.
\end{proposition}

\begin{proof}
Let $1\le k\le N-1$. 
As for the Gauss-Lucas theorem we start from the expression $$\cfrac{P'(X)}{P(X)}=\sum_{i=1}^N\cfrac{1}{X-x_i}$$ in $x_k'$ and the same thing for $Q$. 
We get the following identities: 
\begin{equation}
\left\{ 
\begin{aligned}
0&=\sum_{i=1}^{N}\cfrac{1}{x_k'-x_i}\\
0&=\sum_{i=1}^{N}\cfrac{1}{y_k'-y_i}
\end{aligned}
\label{system1}
\right.
\end{equation}
Subtracting the two equations of \eqref{system1}: \begin{equation}
\label{exprracine}
0=\sum_{i=1}^{N}\cfrac{1}{x_k'-x_i}-\cfrac{1}{y_k'-y_i}=\sum_{i=1}^{N}\cfrac{y'_k-x_k'+x_i-y_i}{(x_k'-x_i)(y_k'-y_i)}.
\end{equation}
We obtain the following identity: $$y'_k-x'_k=\cfrac{\dis\sum_{i=1}^{N}\cfrac{y_i-x_i}{(x_k'-x_i)(y_k'-y_i)}}{\dis\sum_{i=1}^{N}\cfrac{1}{(x_k'-x_i)(y_k'-y_i)}}.$$
Let us remark that for all $1\le k\le N-1$ and for all $1\le i\le N$, $(x_k'-x_i)(y_k'-y_i)\ge 0$ because the two terms of the product have the same sign. Thanks to the hypothesis and the equation \eqref{exprracine} we get that $y_k'\ge x_k'$.
\end{proof}
Even if there is no order on the circle we can state a comparison principle for the roots of trigonometric polynomials.
\newline
Given $N$ points $(x_k)_{0\le k\le N-1}$ in an interval of length $2\pi$ we periodize these points by defining $(x_k)_{k\in\Z}$ as $x_k=x_{k[N]}+2\left\lfloor \cfrac{k}{N}\right\rfloor\pi$ where $k[N]$ is the rest of the euclidean division of $k$ by $N$.
\begin{proposition}[Periodic discrete comparison principle]
\label{prop:discrete comp principle}
Let $x_0<x_1<...<x_{2N-1}< x_0+2\pi$, $y_0<y_1<...<y_{2N-1}< y_0+2\pi$ be two families of $2N$ real numbers and consider the trigonometric polynomials $P_N(X)=\prod_{i=0}^{2N-1}\sin\left(\frac{X-x_i}{2}\right)$ and $Q_N(X)=\prod_{i=0}^{2N-1}\sin\left(\frac{X-y_i}{2}\right)$. We consider the periodic families $(x_i)_{i\in\Z}$ and $(y_i)_{i\in\Z}$ associated to the $(x_i)_{0\le i\le 2N-1}$ and $(y_i)_{0\le i\le 2N-1}$. Let $(x_i')_{i\in\Z}$ and $(y_i')_{i\in\Z}$ be such that $x_i'$ is the only root of the trigonometric polynomial $P_N'$ in $]x_i,x_{i+1}[$ and $y_i'$ is the only root of the trigonometric polynomial $Q_N'$ in $]y_i,y_{i+1}[$.
Assume that for all $i\in\Z$, $x_i\le y_i$, then for all $i\in\Z$, $x_i'\le y_i'$.
\end{proposition}

\begin{proof}
Let $k\in\Z$. 
We follow the proof of the real case. We evaluate $$\cfrac{P'(X)}{P(X)}=\cfrac{1}{2}\sum_{i=0}^{2N-1}\cotan\left(\frac{X-x_i}{2}\right)$$ in $x_k'$ and the same thing for $Q$. 
We get the following identities: 
\begin{equation}
\left\{ 
\begin{aligned}
0&=\sum_{i=0}^{2N-1}\cotan\left(\frac{x_k'-x_i}{2}\right)\\
0&=\sum_{i=0}^{2N-1}\cotan\left(\frac{y_k'-y_i}{2}\right)
\end{aligned}
\label{system2}
\right.
\end{equation}
Since the $(x_i)_{i\in\Z}$ and $(y_i)_{i\in\Z}$ are periodic we can rewrite the previous equations as: 
\begin{equation}
\left\{ 
\begin{aligned}
0&=\sum_{i=1}^{2N}\cotan\left(\frac{x_k'-x_{k+i}}{2}\right)\\
0&=\sum_{i=1}^{2N}\cotan\left(\frac{y_k'-y_{k+i}}{2}\right)
\end{aligned}
\label{system3}
\right.
\end{equation}
Subtracting the two equations of \eqref{system3}: \begin{equation}
\label{exprracine2}
0=\sum_{i=1}^{2N}\left[\cotan\left(\frac{x_k'-x_{k+i}}{2}\right)-\cotan\left(\frac{y_k'-y_{k+i}}{2}\right)\right]
\end{equation}
For all $i\in\Z$, let $a_i=(x_k'-x_{k+i})/2$ and $b_i=(y_k'-y_{k+i})/2$. By definition, for all $i\in\{1,...2N\}$, $a_i,b_i\in(-\pi,0)$. Since $\cotan$ is decreasing on $(-\pi,0)$ and $$0=\sum_{i=1}^{2N}\left[\cotan(a_i)-\cotan(b_i)\right],$$ there exists $i$ such that $b_i\ge a_i$. For this $i$ we have $y_k'-y_{k+i}\ge x_k'-x_{k+i}$ which implies $y_k'-x'_k\ge y_{k+i}-x_{k+i}\ge 0$ by hypothesis.
\end{proof}

\subsection{Convergence of the system of particles}
In this section we derive the solutions of the PDE as a limit of a system of particles justifying the heuristic of Section \ref{section: heuristic pde} for the trigonometric case. The problem is technical and difficult to obtain. In \cite{kiselev2022} authors succeed to obtain a notion of convergence of a system of particle towards the smooth solution of the PDE under some assumptions at initial time on the particles and the smoothness of the initial data of the PDE. Quite often the theory of viscosity solutions allows to avoid such assumptions by an extensive use of discrete comparison principles \cite{barles1991convergence}. The main point is to use the discrete comparison principle and the translation invariance of the problem to obtain a non expansive property of the flow of the differentiation. This allows us to have a control of the error created by repeated differentiation. Let us also mention that this approach follows the idea of the proof of a similar result obtained in the study of the eigenvalues of large random matrices \cite{bertucci2022spectral,bertucci2025new}.
\newline
\tab For $(\lambda_1,...,\lambda_{2N})\in\T^{2N}$ we introduce the empirical measure $\mu_{N}\in\mesT$ associated to the tuple $(\lambda_1,...,\lambda_{2N})$ $$\mu_{N}=\cfrac{1}{2N}\sum_{i=1}^{2N}\delta_{\lambda_i}\in\mesT.$$
This definition shall also be used if we consider a $2N$ periodic system of particles $(\lambda_i)_{i\in \Z}\in\T^{\Z}$ (i.e. for all $i\in\Z$, $\lambda_i=\lambda_{i[2N]}$ modulo $2\pi$ with $i[2N]$ the rest of the euclidean division of $i$ by $2N$) by defining the empirical measure on $\mesT$ associated to these particles as $$\mu_{N}=\cfrac{1}{2N}\sum_{i\in\Z/2N\Z}\delta_{\lambda_i}\in\mesT.$$
\newline
\tab For $a\in\R$ (that be specified later on), let $a<\lambda_0^1<...<\lambda_0^{2N-1}<\lambda_0^{2N}<a+2\pi$ be $2N$ initial particles. We associate to this family the trigonometric polynomial $$p_{2N}(x)=\prod_{j=0}^{2N-1}\sin\left(\cfrac{x-\lambda_0^j}{2}\right).$$ We define for $i\in \Z$ the $i^{th}$ initial particle $\lambda_0^i=\lambda_0^{i[2N]}+2\left\lfloor \frac{i}{2N}\right\rfloor\pi$. 
\newline
We construct the following particles $(\lambda_t^{i})_{i\in\Z,t\ge 0}$ by : for $t\in[0,\frac{1}{2N})$, $\lambda_t^{i}=\lambda_0^{i}$ and  for $k\in\N_{\ge 1}$ and $t\in \left[\frac{k}{2N},\frac{k+1}{2N}\right)$, $\lambda_t^{i}$ is the unique root of $p_{2N}^{(k)}$ in $\left(\lambda_{\frac{k-1}{2N}}^{i},\lambda_{\frac{k-1}{2N}}^{i+1}\right)$. Formally, we start from the  family $(\lambda_0^i)_{i\in\Z}$ which are the roots of $p_{2n}$, we wait a time $\Delta t_N=\frac{1}{2N}$, then at the time $t=\frac{1}{2N}$, the particles jump to the roots of the derivative of $p_{2N}$ and we repeat the procedure. 
\newline
For all $t\ge 0$, for all $\theta\in\R$, we define $$F_{N}(t,\theta):=\cfrac{1}{2N}\max\{j\in\Z\,:\, \lambda_t^j\le \theta\}.$$ 
The goal of this section is to establish the convergence of the system of particles using the convergence of the primitive functions $F_{N}$.
\begin{theorem}
\label{theorem:convparticles}
Assume that the empirical measure $\mu_N^0$ of initial conditions defined by: $$\mu_N^0=\frac{1}{2N}\sum_{i=0}^{2N-1}\delta_{\lambda_0^i}$$ 
converges toward a measure $\mu_0\in\mesT$. Let $F_0(x):=\mu_0((a,x])\in \mathcal F_{2\pi}$ with $a\in\R$ a point that is not an atom of $\mu_0$ and of $\mu_N^0$ for all $N\ge 1$.
\newline
Then we have the uniform in time convergence $$ \sup_{0\le t\le T} W_1(\mu_N(t),\partial_x F(t,.))\underset{N\to\infty}{\longrightarrow} 0\,,$$
where $F$ is the unique viscosity solution of \eqref{eqPrimitivePol} with initial condition $F_0$.
\end{theorem}
The proof of Theorem \ref{theorem:convparticles} is given in two parts. We first study the convergence of the primitive $F_N$ toward $F$, the viscosity solution, and we finally complete the proof of Theorem \ref{theorem:convparticles}.

\subsubsection{Convergence of the primitives}
Let us first give some notation. 

We write $X^N=(X_i^N)_{i\in\Z}$ for (periodic) $2N$-tuple such that $X_1^N<X_2^N...<X_{2N}^N<X_{2N+1}=X_{1}^N+2\pi$. Given a such periodic $2N$-tuple we introduce its cumulative distribution $$F_{X^N}:=\frac{1}{2N}\max\{i\in\Z\,|\, X_i^{N}\le x\}.$$

We also denote the derivation of a $2N$-tuple by $D_N(X^N)$ which is the periodic $2N-$tuple of the roots of the derivative of the trigonometric polynomial of degree $2N$ with roots $X^N$ (which are labeled such that $(D_N(X^N))_i\in(X_i^N,X_{i+1}^N)$). 
Obviously, we have the following property for all $b\in\R$ and $X^N$: $D_N(X^N+b)=D_N(X^N)+b$. Combining this observation with the discrete comparison principle, we obtain for all $X^N$ and $Y^N$ a non expansive property of the derivation operator:
\begin{equation}
\label{eq:contractionofderivation}
\|D_N(X^N)-D_N(Y^N)\|_\infty\le \|X^N-Y^N\|_\infty.
\end{equation}
In this section, for every $N\ge 1$, let $(X^{0,N}_i)_{i\in\Z}$ a periodic $2N-$tuple. For $k\ge 1$, we write $X^{k,N}:=D^k_N(X^N)$ the $k$ iterative of the operator $D_N$ on $X^N$. 
We also introduce, as before, the empirical distribution associated to these tuples $$G_N(t,x):=F_{X^{\lfloor 2Nt\rfloor,N}}(x).$$
Given $F$ and $G$ two non-decreasing right-continuous functions from $\R\to \R$ we recall the Lévy distance $$d_L(F,G)=\inf\{r>0\,|\,F(x-r)-r\le G(x)\le F(x+r)+r\}.$$
This distance naturally appears in probability since a sequence of real probability measures converges toward a probability measure if and only the sequence of its cumulative distribution functions converges for the Lévy distance toward the cumulative distribution function of the limit measure. Moreover, we always have $d_L(F,G)\le \|F-G\|_\infty$.

In all this section, we also write $t_k=\frac{k}{2N}$ the time discretization and we fix an horizon of time $T>0$.

\begin{theorem}
\label{theorem:convempirique}
Let $F_0$ be a bounded non-decreasing right-continuous function such that $$e_N:=d_L(G_N(0,.),F_0)\underset{N\to+\infty}{\longrightarrow}0.$$
Then, we have the following convergence $$\sup_{t\in[0,T]} d_L(G_N(t,.),F(t,.))\underset{N\to+\infty}{\longrightarrow}0,$$
where $F$ is the unique viscosity solution of \eqref{eqPrimitivePol} with initial condition $F_0$.
\end{theorem}
The proof of this theorem is cut in different steps. 
\paragraph{Non expansion of Lévy distance by differentiation of tuples}
\label{para:contraction}
By definition of the cumulative distribution of $2N$-tuples, for $X^N$ and $Y^N$ two tuples, we have that for all $h\in\R$ and $l\in\Z$, $F_{X^N}(x)\le F_{Y^N}(x+h)+\frac{l}{2N}$ for all $x\in \R$ if and only if $Y_{i-l}^N-h\le X_i$ for all $i\in \Z$.

Using this observation and the discrete comparison principle immediately yields the following result.
\begin{lemma}
\label{lemma:contractionprincipleparticles}
For $X^N$ and $Y^N$ two periodic $2N-$tuples we have $$d_L(F_{D_N(X^N)},F_{D_N(Y^N)})\le d_L(F_{X^N},F_{Y^N}).$$
\end{lemma}
\paragraph{Auxiliary system}
\label{para:approxsystem}
 In this part, we consider, using Theorem 1.1 of \cite{kiselev2022}, $U$ a smooth solution of \eqref{eqPrimitivePol} with $U(t,x+2\pi)=U(t,x)+1$ for every $t\ge 0$, $x\in\R$ and such that for all $x\in\R$ and $t\ge 0$ $$0<m\le u(t,x)=\partial_x U(t,x)\le M.$$
We shall construct some particles associate to $U$ and then study the flow created by the differentiation on these particles as presented in the heuristic in Section \ref{section: heuristic pde}. Using the smoothness of $U$ we shall control the effect of the derivation on the particles. We introduce the quantity that appears in the limit PDE $$\Gamma(t,x):=\frac{1}{\pi}\left(\frac{\pi}{2}+\arctan\left(\frac{H[u(t)](x)}{u(t,x)}\right)\right)\in (0,1).$$
By hypothesis on $U$, we can define $Q:\R^+\times\R\to\R$ the quantile function associated to $U$ by $U(t,Q(t,x))=x$ for every $x\in\R$ and  construct particles associated to this quantile function as follows $(Q^N(t))_i:=Q\left(t,\frac{i}{2N}\right).$ 

Let us mention that the quantile function is smooth and satisfy $\frac{1}{M}\le\partial_x Q(t,x)\le \frac{1}{m}$.
The next lemma is the key point to understand the error created by the differentiation.
\begin{lemma}
\label{lemma:errordifferentiation}
There exists a constant $K$ that only depends on $U$, its derivatives and $T$ such that for all $N\ge 1$ 
$$\sup_{t\in[0,T]}\left\|D_N(Q^N(t))-Q^N\left(t+\frac{1}{2N}\right)\right\|_\infty\le \cfrac{K}{N^2} \,.$$
\end{lemma}
\begin{proof}
In all this proof, $K$ will denote a constant that only depends on $U$, its derivatives and $T$. We shall used the fundamental identity, as in the heuristic of Section \ref{section: heuristic pde}, which is that if $z\in D_N(Q^N(t))$ we have 
\begin{equation}
\label{eq:sommenulpolynometrigo}
\cfrac{1}{2N}\sum_{j=0}^{2N-1}\cotan\left(\frac{z-(Q^N(t))_j}{2}\right)=0.
\end{equation}
We first fix $z\in\R$ and $t\in[0,T]$ and write $b:=U(t,z)$ and define the function $$R_{t,z}(a):=\cotan\left(\frac{z-Q(t,a)}{2}\right)-2\pi u(t,z)\cotan(\pi(b-a)).$$
This function is smooth in all the parameters expect when $a=b$ modulo $1$. Noticing that $u(t,z)=1/Q'(b)$, doing a Taylor expansion of $R_{t,z}$ in $a=b$ modulo $1$, we can show that $R_{t,z}$ is actually $C^1$ everywhere and smooth in all the parameters. Hence we get $$K:=\sup_{t\in[0,T],z\in\R}\|\partial_a R_{t,z}\|_\infty<+\infty.$$
Now, we write for every $t\ge 0,\, z\in\R$
\begin{align*}
\cfrac{1}{2N}\sum_{i=0}^{2N-1}\cotan\left(\frac{z-(Q^N(t))_j}{2}\right)=\cfrac{1}{2N}\sum_{i=0}^{2N-1}&R_{t,z}\left(\frac{j}{2N}\right)+\\
&\cfrac{2\pi u(t,z)}{2N}\sum_{i=0}^{2N-1}\cotan\left(\pi\left(b-\frac{j}{2N}\right)\right).
\end{align*}
Recalling the trigonometric identity $$\sum_{i=0}^{2N-1}\cotan\left(\pi\left(b-\frac{j}{2N}\right)\right) =2N\cotan(\pi 2N b),$$ gives 
\begin{align*}
\cfrac{1}{2N}\sum_{i=0}^{2N-1}\cotan\left(\frac{z-(Q^N(t))_j}{2}\right)=\cfrac{1}{2N}\sum_{i=0}^{2N-1}R_{t,z}\left(\frac{j}{2N}\right)+2\pi u(t,z)\cotan(\pi 2N U(t,z)).
\end{align*}
We compute the Riemann sum using the uniform bound on $\partial_a R_{t,z}$: 
$$|f_N|:=\left|\cfrac{1}{2N}\sum_{i=0}^{2N-1}R_{t,z}\left(\frac{j}{2N}\right)-\int_0^1 R_{t,z}(a)da\right|\le \frac{K}{N}$$
to obtain 
\begin{align*}
\cfrac{1}{2N}\sum_{i=0}^{2N-1}\cotan\left(\frac{z-(Q^N(t))_j}{2}\right)=\int_0^1 R_{t,z}(a)da+&f_N+2\pi u(t,z)\cotan(\pi 2N U(t,z)),
\end{align*}
with $|f_N|\le \frac{K}{N}$ for all $N\ge 1$.
Finally, using that $\cotan$ is an odd function we have 
\begin{align*}
\int_0^1 R_{t,z}(a)da&=\text{P.V.}\int_0^1 \cotan\left(\frac{z-Q(t,a)}{2}\right)\,da\\
&=\text{P.V.}\int_{Q(t,0)}^{Q(t,1)} \cotan\left(\frac{z-b}{2}\right)u(t,b)\,db,\\
&=2\pi H[u(t)](z)
\end{align*}
doing the change of variable $a=U(t,b)$.
Hence, we get that for all $t\ge 0$, $z\in\R$ \begin{align*}
\cfrac{1}{2N}\sum_{i=0}^{2N-1}\cotan\left(\frac{z-(Q^N(t))_j}{2}\right)=2\pi H[u(t)](z)+f_N+2\pi u(t,z)\cotan(\pi 2N U(t,z)),
\end{align*}
with $|f_N|\le \frac{K}{N}$ for all $N\ge 1$.

Using \eqref{eq:sommenulpolynometrigo} in $z_i=(D_N(Q^N(t)))_i$ we have $$0=2\pi H[u(t)](z_i)+f_N+2\pi u(t,z_i)\cotan(\pi 2N U(t,z_i)).$$ which implies 
$$\cotan(\pi 2N U(t,z_i))=-\cfrac{H[u(t)](z_i)}{u(t,z_i)}-\cfrac{f_N}{2\pi u(t,z_i)}.$$
By the interlacing property of the roots of the derivative of a polynomial, we get that $0<2NU(t,z_i)-i<1$ and $$\cotan(\pi (2N U(t,z_i)-i))=-\cfrac{H[u(t)](z_i)}{u(t,z_i)}-\cfrac{f_N}{2\pi u(t,z_i)}.$$
Moreover, by definition of $\Gamma(t,z)\in(0,1)$ we also have $$\cotan(\pi\Gamma(t,z_i))=-\cfrac{H[u(t)](z_i)}{u(t,z_i)}.$$
Since on $(0,\pi)$ $|\cotan'|\ge 1$ the mean value inequality yields $$|2N U(t,z_i)-i-\Gamma(t,z_i)|\le \frac{|f_N|}{2\pi^2u(t,z_i)}.$$
Dividing by $2N$ gives 
\begin{equation}
\label{ineqbound1}
\left|U(t,z_i)-\frac{i}{2N}-\frac{\Gamma(t,z_i)}{2N}\right|\le \frac{|f_N|}{4N\pi^2u(t,z_i)}\le \frac{K}{N^2},
\end{equation}
with $K$ a new constant that only depends on $U$, its derivatives and $T$.
Using the smoothness of $U$ there exists a constant $K>0$, that depends on $U$, its derivatives and $T$, such that for all $t\ge 0$, $i\in\Z$, $$U\left(t+\frac{1}{2N},z_i\right)=U(t,z_i)+\frac{1}{2N}\partial_t U(t,z_i)+\frac{g_{N,t,i}}{N^2}=U(t,z_i)-\frac{1}{2N}\Gamma(t,z_i)+\frac{g_{N,t,i}}{N^2},$$
with $|g_{N,t,i}|\le K.$
Hence, we get 
\begin{equation}
\label{ineqbound2}
\left|U\left(t+\frac{1}{2N},z_i\right)-\frac{i}{2N}\right|\le \frac{K}{N^2},
\end{equation}
where $K>0$ is a new constant which only depends on $U$, its derivatives and $T$.
Using that $\|\partial_x Q\|_\infty\le \frac{1}{m}$, the mean value theorem yields 
\begin{align*}
\left|z_i-Q\left(t+\frac{1}{2N},\frac{i}{2N}\right)\right| &\le \frac{1}{m}\left|U\left(t+\frac{1}{2N},z_i\right)-\frac{i}{2N}\right|\\
&\le \frac{K}{N^2},
\end{align*}
with $K$ a new constant that again only depends on $U$, its derivatives and $T$. This proves the result.
\end{proof}
We can now define another system of particles associated to the function $U$. We first write $Y_j^{0,N}:=Q\left(0,\frac{j}{2N}\right)$ and we then define for every $k\ge 0$ $Y^{k,N}=D^k_N(Y^{0,N})$ which corresponds to the particles induced by the flow of the differentiation. We also write $Y^N(t)=Y^{k,N}$ if $\lfloor2N t\rfloor=k$.
Using the previous lemma and the non expansion property of the differentiation, we can evaluate the error between this system of particles and the quantile of $U$.
\begin{lemma}
\label{lemma:errorterm}
We define an error term as follows $E_k:=\|Y^{k,N}-Q^N(t_k)\|_\infty$. There exists $K>0$ that only depends on $U$, its derivatives, and $T$ such that for all $N\ge 1$ and $k\ge 1$ such that $t_k\le T$ \begin{equation}
\label{bounderror}
E_k\le \frac{K}{N}.
\end{equation}
Moreover, there exists $K>0$ such that for every $N\ge 1$ $$\sup_{t\in[0,T]}\|F_{Y^N(t)}-U(t,.)\|_\infty\le \frac{K}{N}.$$
\end{lemma}
\begin{proof}
We directly have 
\begin{align*}
E_{k+1}&=\|D_N(Y^{k,N})-Q^N(t_{k+1})\|_\infty\\
&\le\|D_N(Y^{k,N})-D_N(Q^N(t_k))\|_\infty+\|D_N(Q^N(t_k))-Q^N(t_{k+1})\|_\infty\\
&\le \|Y^{k,N}-Q^N(t_k)\|_\infty +\frac{K}{N^2}=E_k+\frac{K}{N^2},
\end{align*}
where the last inequality comes from the non expansion of the derivation \eqref{eq:contractionofderivation} and Lemma \ref{lemma:errordifferentiation}.
Summing these relations yields for every $0\le k\le \lfloor 2NT\rfloor$, $$E_k\le \frac{k K}{N^2}\le \frac{2TK}{N},$$
giving the first result. 
For the second inequality, we introduce the cumulative distribution associated to $Q^N$: $\overline U_N(t,.)=F_{Q^N(t)}(.)$. This function satisfied 
\begin{equation}
\label{eq:boundapprox}
0\le U-\overline U_N\le \frac{1}{2N} \,.
\end{equation} Moreover, by definition of the error term, for every $k\ge 0$ $$\overline U_N(t_k,x-E_k)\le F_{Y^N(t_k)}(x)\le \overline U_N(t_k,x+E_k)$$ and using $|\partial_t U|\le 1$ and \eqref{eq:boundapprox}, we deduce that for all $t\in[0,T]$ $$\|F_{Y^N(t)}-U(t,.)\|_\infty\le M E_{\lfloor2N t\rfloor}+\frac{1}{N},$$
concluding the proof.
\end{proof}
\begin{remark}
This result is a result of propagation of error as obtained in Theorem 1.2 in \cite{kiselev2022}. Here the proof is easier thanks to the use of discrete comparison principle giving a non expansion property that allows to have a control of the error made after one differentiation.
\end{remark}
\paragraph{Approximation of viscosity solutions by smooth solutions}

Let $\rho$ be smooth non negative even compactly supported function in $(-1,1)$ of mass $1$ and define $\rho_\delta(x)=\delta^{-1}\rho(x/\delta)$.
We regularize and approximate the initial condition as follows: for $\delta>0$ and $0<\varepsilon<1$, $h_\delta=\rho_\delta \ast F_0$ and $ f_{\varepsilon,\delta}=(1-\varepsilon)h_\delta+\varepsilon \frac{x}{2\pi}$. Introducing the quantity $\|F_0-\frac{x}{2\pi}\|_\infty=:C_0$ we have $\| f_{\varepsilon,\delta}-h_\delta\|\le C_0\varepsilon$ and $\partial_x f_{\varepsilon,\delta}\ge \frac{\varepsilon}{2\pi}$. 
By Theorem $1.1$ of \cite{kiselev2022}, we can consider a smooth solution $F^{\varepsilon,\delta}$ of \eqref{eqPrimitivePol} with initial condition $f_{\varepsilon,\delta}$.
\newline
Using the comparison principle for viscosity solutions of \eqref{eqPrimitivePol} we can compare $F^{\varepsilon,\delta}$ and $F$ for every time
\begin{lemma}
\label{lemma:comparisonfunctions}
For every $t\ge 0$ and $x\in\R$ we have $$F^{\varepsilon,\delta}(t,x-\delta)-C_0\varepsilon\le F_*(t,x)\le F(t,x)\le F^{\varepsilon,\delta}(t,x+\delta)+C_0\varepsilon.$$
As a consequence we get $$\sup_{t\ge 0}d_L(F^{\varepsilon,\delta}(t,.),F(t,.))\le \delta+C_0\varepsilon.$$
\end{lemma}
\begin{proof}
Since $\rho_\delta$ is supported in $(-\delta,\delta)$ and $F_0$ is non-decreasing, we get $$h_\delta(x-\delta)\le F_0(x)\le h_\delta(x+\delta).$$
By continuity of $h_\delta$ we have \begin{align*}
f_{\varepsilon,\delta}(x-\delta)-C_0\varepsilon\le h_\delta(x-\delta)\le (F_0)_*(x)\le F_0(x)\le h_\delta(x+\delta)\le f_{\varepsilon,\delta}(x+\delta)+C_0\varepsilon.
\end{align*}
For the first inequality, we note that since the initial condition of $F$ is $F_0$ $$f_{\varepsilon,\delta}(x-\delta)-C_0\varepsilon\le (F_0)_*(x)\le F_*(0,x).$$ The comparison principle of Theorem \ref{thm:compprinciplecomplet} yields $$F^{\varepsilon,\delta}(t,x-\delta)-C_0\varepsilon\le F_*(t,x).$$ We similarly prove the second inequality.
\end{proof}

\paragraph{Conclusion and convergence}
We now conclude the proof of Theorem \ref{theorem:convempirique}. 

We define $(Y^{N}_{\varepsilon,\delta}(t))_{t\in[0,T]}$ as in Paragraph \ref{para:approxsystem} associated to the function $F^{\varepsilon,\delta}$. By construction, 
$$0\le f_{\varepsilon,\delta}-F_{Y^{N}_{\varepsilon,\delta}(0)}\le \frac{1}{2N}$$ which implies that $$d_L(F_{Y^{N}_{\varepsilon,\delta}(0)},f^{\varepsilon,\delta})\le \frac{1}{2N}.$$
Moreover, by Lemma \ref{lemma:errorterm}, we can find $K_{\varepsilon,\delta}$ which depends on $T$ and $U^{\varepsilon,\delta}$ such that 
$$d_L(F_{Y^{N}_{\varepsilon,\delta}(t)},F^{\varepsilon,\delta}(t,.))\le \frac{K_{\varepsilon,\delta}}{N}.$$
Using these estimates, the non expansion of Lévy distance and the approximation of Lemma \ref{lemma:comparisonfunctions}, we get for all $0\le t\le T$
\begin{align*}
d_L(G_N(t,.),F(t,.))&\le d_L(G_N(t,.),F_{Y^{N}_{\varepsilon,\delta}(t)})+d_L(F_{Y^{N}_{\varepsilon,\delta}(t)},F^{\varepsilon,\delta}(t,.))+d_L(F^{\varepsilon,\delta}(t,.),F(t,.))\\
&\le d_L(G_N(0,.),F_{Y^{N}_{\varepsilon,\delta}(0)})+\frac{K_{\varepsilon,\delta}}{N}+ \delta+C_0\varepsilon\\
&\le d_L(G_N(0,.),F_0)+d_L(F_0,f_{\varepsilon,\delta})+d_L(f_{\varepsilon,\delta},F_{Y^{N}_{\varepsilon,\delta}(0)})+\\
&\hspace{8.5cm}\frac{K_{\varepsilon,\delta}}{N}+ \delta+C_0\varepsilon\\
&\le e_N+2(\delta+C_0\varepsilon)+\frac{K_{\varepsilon,\delta}+1}{N}\,.
\end{align*}
Thus, we obtain, for all $\varepsilon,\delta$, $$\limsup_{N\to +\infty} \sup_{t\in[0,T]}d_L(G_N(t,.),F(t,.))\le 2(\delta+C_0\varepsilon).$$
Letting $\varepsilon$ and $\delta$ goes to 0 yields the result.
\subsubsection{Conclusion of the proof of Theorem \ref{theorem:convparticles}}
Before proving Theorem \ref{theorem:convparticles}, we explicit a quantitative inequality between the Lévy distance of primitive of probability measures on the circle and the $W_1$ distance. This estimate is possible since we are on a compact set and so the convergence in law of probability measures is equivalent with the convergence for the distance $W_1$.
\begin{lemma}
\label{lemma:boundW_1law}
Given $F$ and $G$ two non-decreasing right-continuous that satisfies $F(.+2\pi)=F(.)+1$ and $G(.+2\pi)=G(.)+1$. Denoting by $\partial_x F\in\mesT$ and $\partial_x G\in\mesT$ their derivatives in the sense of distributions, we have $$W_1(\partial_x F, \partial_x G)\le (2+2\pi)d_L(F,G).$$
\end{lemma}
\begin{proof}
Let $r$ be an admissible radius in the definition of $d_L(F,G)$. We notice that for all $x\in\R$ $$|F(x)-G(x)|\le F(x+r)-F(x-r)+r.$$ Since $F(x)-\frac{x}{2\pi}$ is periodic we also have $$\int_t^{t+2\pi} (F(x+r)-F(x-r))dx=2r.$$
So if $\phi$ is a $1$-Lipschitz function we obtain for every $s\in\R$
\begin{align*}
\left|\int_\T\phi(\partial_x F-\partial_x G)\right|&=\left|\int_s^{s+2\pi}\phi'(x)(F(x)-G(x))dx\right|\\
&\le \int_s^{s+2\pi}|F(x)-G(x)|\le (2+2\pi) r .
\end{align*}
Taking the infimum on $r$ and then the supremum on $\phi$ gives the result by the Kantorovitch duality.
\end{proof}
We can now conclude the proof of Theorem \ref{theorem:convparticles}.
\begin{proof}
By the choice of labeling of the particles at time 0 with respect to $a$, $F_N(0,x)=\mu_N(0)((a,x])$ for $x\in[a,a+2\pi]$ and then the functions are extended by $F_N(0,x+2\pi)=F_N(0,x)+1$. Since $a$ is not an atom of $\mu_0$, the convergence in law of $\mu_N^0$ toward $\mu_0$ implies that $d_L(F_N(0),F_0)$ goes to 0 when $N$ goes to $+\infty$. Since $\partial_x F_N=\mu_N$, Lemma \ref{lemma:boundW_1law} and Theorem \ref{theorem:convempirique} give the result.
\end{proof}
\subsection*{Acknowledgments}
The first author acknowledges a partial support from the Chair FDD (ILB) and the Lagrange Mathematics and Computing Research Center. This work was partially funded by the ERC project PaDiESeM.

\renewcommand{\MR}[1]{}
\bibliographystyle{smfplain}
\bibliography{refpolynome}

\providecommand{\bysame}{\leavevmode ---\ }
\providecommand{\og}{``}
\providecommand{\fg}{''}
\providecommand{\smfandname}{\&}
\providecommand{\smfedsname}{\'eds.}
\providecommand{\smfedname}{\'ed.}
\providecommand{\smfmastersthesisname}{M\'emoire}
\providecommand{\smfphdthesisname}{Th\`ese}
\begin{thebibliography}{10}

\bibitem{angst2024almost}
{\scshape J.~Angst, D.~Malicet {\normalfont \smfandname} G.~Poly} -- {\og
  Almost sure behavior of the critical points of random polynomials\fg},
  \emph{Bulletin of the London Mathematical Society} \textbf{56} (2024), no.~2,
  p.~767--782.

\bibitem{arisawa2006}
{\scshape M.~Arisawa} -- {\og A new definition of viscosity solutions for a
  class of second-order degenerate elliptic integro-differential equations\fg},
  in \emph{Annales de l'Institut Henri Poincar{\'e} C, Analyse non
  lin{\'e}aire}, vol.~23, Elsevier, 2006, p.~695--711.

\bibitem{arisawa2008}
\bysame , {\og A remark on the definitions of viscosity solutions for the
  integro-differential equations with {L}{\'e}vy operators\fg}, \emph{Journal
  de math{\'e}matiques pures et appliqu{\'e}es} \textbf{89} (2008), no.~6,
  p.~567--574.

\bibitem{barles2008}
{\scshape G.~Barles {\normalfont \smfandname} C.~Imbert} -- {\og Second-order
  elliptic integro-differential equations: viscosity solutions' theory
  revisited\fg}, in \emph{Annales de l'IHP Analyse non lin{\'e}aire}, vol.~25,
  2008, p.~567--585.

\bibitem{barles1991convergence}
{\scshape G.~Barles {\normalfont \smfandname} P.~E. Souganidis} -- {\og
  Convergence of approximation schemes for fully nonlinear second order
  equations\fg}, \emph{Asymptotic analysis} \textbf{4} (1991), no.~3,
  p.~271--283.

\bibitem{bertucci2022spectral}
{\scshape C.~Bertucci, M.~Debbah, J.-M. Lasry {\normalfont \smfandname} P.-L.
  Lions} -- {\og A spectral dominance approach to large random matrices\fg},
  \emph{Journal de Math{\'e}matiques Pures et Appliqu{\'e}es} \textbf{164}
  (2022), p.~27--56.

\bibitem{bertucci2024spectral}
{\scshape C.~Bertucci, J.-M. Lasry {\normalfont \smfandname} P.-L. Lions} --
  {\og A spectral dominance approach to large random matrices: part ii\fg},
  \emph{Journal de Math{\'e}matiques Pures et Appliqu{\'e}es} \textbf{192}
  (2024), p.~103630.

\bibitem{bertucci2025new}
{\scshape C.~Bertucci {\normalfont \smfandname} V.~Pesce} -- {\og A new
  approach for the unitary {D}yson {B}rownian motion through the theory of
  viscosity solutions\fg}, \emph{arXiv preprint arXiv:2504.16551} (2025).

\bibitem{bloch1932roots}
{\scshape A.~Bloch {\normalfont \smfandname} G.~P{\'o}lya} -- {\og On the roots
  of certain algebraic equations\fg}, \emph{Proceedings of the London
  Mathematical Society} \textbf{2} (1932), no.~1, p.~102--114.

\bibitem{crandall1992}
{\scshape M.~G. Crandall, H.~Ishii {\normalfont \smfandname} P.-L. Lions} --
  {\og User’s guide to viscosity solutions of second order partial
  differential equations\fg}, \emph{Bulletin of the American mathematical
  society} \textbf{27} (1992), no.~1, p.~1--67.

\bibitem{curgus2004contraction}
{\scshape B.~{\'C}urgus {\normalfont \smfandname} V.~Mascioni} -- {\og A
  contraction of the {L}ucas polygon\fg}, \emph{Proceedings of the American
  Mathematical Society} \textbf{132} (2004), no.~10, p.~2973--2981.

\bibitem{do2015real}
{\scshape Y.~Do, H.~Nguyen {\normalfont \smfandname} V.~Vu} -- {\og Real roots
  of random polynomials: expectation and repulsion\fg}, \emph{Proceedings of
  the London Mathematical Society} \textbf{111} (2015), no.~6, p.~1231--1260.

\bibitem{do2023random}
{\scshape Y.~Do, O.~Nguyen {\normalfont \smfandname} V.~Vu} -- {\og Random
  orthonormal polynomials: Local universality and expected number of real
  roots\fg}, \emph{Transactions of the American Mathematical Society}
  \textbf{376} (2023), no.~09, p.~6215--6243.

\bibitem{dunnage1966number}
{\scshape J.~Dunnage} -- {\og The number of real zeros of a random
  trigonometric polynomial\fg}, \emph{Proceedings of the London Mathematical
  Society} \textbf{3} (1966), no.~1, p.~53--84.

\bibitem{dyson1962}
{\scshape F.~J. Dyson} -- {\og A {B}rownian-motion model for the eigenvalues of
  a random matrix\fg}, \emph{Journal of Mathematical Physics} \textbf{3}
  (1962), no.~6, p.~1191--1198.

\bibitem{edelman1995many}
{\scshape A.~Edelman {\normalfont \smfandname} E.~Kostlan} -- {\og How many
  zeros of a random polynomial are real?\fg}, \emph{Bulletin of the American
  Mathematical Society} \textbf{32} (1995), no.~1, p.~1--37.

\bibitem{farmer2006crystallization}
{\scshape D.~W. Farmer {\normalfont \smfandname} M.~Yerrington} -- {\og
  Crystallization of random trigonometric polynomials\fg}, \emph{Journal of
  statistical physics} \textbf{123} (2006), p.~1219--1230.

\bibitem{galligo2024anti}
{\scshape A.~Galligo, J.~Najnudel {\normalfont \smfandname} T.~Vu} -- {\og
  Anti-concentration applied to roots of randomized derivatives of
  polynomials\fg}, \emph{Electronic Journal of Probability} \textbf{29} (2024),
  p.~1--20.

\bibitem{hall2023zeros}
{\scshape B.~C. Hall, C.-W. Ho, J.~Jalowy {\normalfont \smfandname}
  Z.~Kabluchko} -- {\og Zeros of random polynomials undergoing the heat
  flow\fg}, \emph{arXiv preprint arXiv:2308.11685} (2023).

\bibitem{hough2009zeros}
{\scshape J.~B. Hough, M.~Krishnapur, Y.~Peres {\normalfont et~al.}} --
  \emph{Zeros of gaussian analytic functions and determinantal point
  processes}, vol.~51, American Mathematical Soc., 2009.

\bibitem{jalowy2025zeros}
{\scshape J.~Jalowy, Z.~Kabluchko {\normalfont \smfandname} A.~Marynych} --
  {\og Zeros and exponential profiles of polynomials i: Limit distributions,
  finite free convolutions and repeated differentiation\fg}, \emph{arXiv
  preprint arXiv:2504.11593} (2025).

\bibitem{kabluchko2015critical}
{\scshape Z.~Kabluchko} -- {\og Critical points of random polynomials with
  independent identically distributed roots\fg}, \emph{Proceedings of the
  American Mathematical Society} \textbf{143} (2015), no.~2, p.~695--702.

\bibitem{kabluchko2021repeated}
\bysame , {\og Repeated differentiation and free unitary poisson process\fg},
  \emph{arXiv preprint arXiv:2112.14729} (2021).

\bibitem{kabluchko2022lee}
\bysame , {\og Lee-yang zeroes of the {C}urie-{W}eiss ferromagnet, unitary
  {H}ermite polynomials, and the backward heat flow\fg}, \emph{arXiv preprint
  arXiv:2203.05533} (2022).

\bibitem{kac1943average}
{\scshape M.~Kac} -- {\og On the average number of real roots of a random
  algebraic equation\fg},  (1943).

\bibitem{kiselev2022}
{\scshape A.~Kiselev {\normalfont \smfandname} C.~Tan} -- {\og The flow of
  polynomial roots under differentiation\fg}, \emph{Annals of PDE} \textbf{8}
  (2022), no.~2, p.~16.

\bibitem{littlewood1938number}
{\scshape J.~E. Littlewood {\normalfont \smfandname} A.~C. Offord} -- {\og On
  the number of real roots of a random algebraic equation\fg}, \emph{Journal of
  the London Mathematical Society} \textbf{1} (1938), no.~4, p.~288--295.

\bibitem{michelen2024almost}
{\scshape M.~Michelen {\normalfont \smfandname} X.-T. Vu} -- {\og Almost sure
  behavior of the zeros of iterated derivatives of random polynomials\fg},
  \emph{Electronic Communications in Probability} \textbf{29} (2024), p.~1--10.

\bibitem{michelen2024zeros}
\bysame , {\og Zeros of a growing number of derivatives of random polynomials
  with independent roots\fg}, \emph{Proceedings of the American Mathematical
  Society} \textbf{152} (2024), no.~06, p.~2683--2696.

\bibitem{pemantle2013distribution}
{\scshape R.~Pemantle {\normalfont \smfandname} I.~Rivin} -- {\og The
  distribution of zeros of the derivative of a random polynomial\fg}, in
  \emph{Advances in Combinatorics: Waterloo Workshop in Computer Algebra, W80,
  May 26-29, 2011}, Springer, 2013, p.~259--273.

\bibitem{steinerberger2019nonlocal}
{\scshape S.~Steinerberger} -- {\og A nonlocal transport equation describing
  roots of polynomials under differentiation\fg}, \emph{Proceedings of the
  American Mathematical Society} \textbf{147} (2019), no.~11, p.~4733--4744.

\bibitem{stevens1969average}
{\scshape D.~C. Stevens} -- {\og The average number of real zeros of a random
  polynomial\fg}, \emph{Communications on Pure and Applied Mathematics}
  \textbf{22} (1969), no.~4, p.~457--477.

\bibitem{stoyanoff1925theoreme}
{\scshape A.~Stoyanoff} -- {\og Sur un th{\'e}or{\`e}me de {M}. {M}arcel
  {R}iesz\fg}, \emph{Nouvelles annales de math{\'e}matiques: journal des
  candidats aux {\'e}coles polytechnique et normale} \textbf{1} (1925),
  p.~97--99.

\bibitem{szeg1939orthogonal}
{\scshape G.~Szeg} -- \emph{Orthogonal polynomials}, vol.~23, American
  Mathematical Soc., 1939.

\end{thebibliography}
\end{document}